\documentclass[10pt, reqno]{amsart}
\usepackage{amsmath, amssymb, amsfonts, amsthm, graphics, mathrsfs}
\usepackage{setspace}
\usepackage{mathtools, bbm}
\usepackage{mathdots, color}
\usepackage[all]{xy}
\usepackage[colorlinks=true,breaklinks,pagebackref]{hyperref}

\usepackage[lite,abbrev,alphabetic]{amsrefs}
\usepackage[foot]{amsaddr}
\usepackage{lineno}

\theoremstyle{definition}
\newtheorem{dfn}{Definition}
\newtheorem{thm}[dfn]{Theorem}
\newtheorem{lem}[dfn]{Lemma}

\newtheorem{rem}[dfn]{Remark}
\newtheorem{cor}[dfn]{Corollary}

\newtheorem{condition}[dfn]{Condition}

\def\MR#1{\quad \href{http://www.ams.org/mathscinet-getitem?mr=#1}{MR#1}}

\def\N{\mathbb{N}}
\def\C{\mathbb{C}}
\def\R{\mathbb{R}}
\def\Q{\mathbb{Q}}
\def\Z{\mathbb{Z}}
\def\A{\mathbb{A}}

\def\GL{{\mathop{\mathrm{GL}}}}
\def\PGL{{\mathop{\mathrm{PGL}}}}
\def\SL{{\mathop{\mathrm{SL}}}}
\def\O{{\mathop{\textnormal{O}}}}
\def\SO{{\mathop{\mathrm{SO}}}}
\def\U{{\mathop{\textnormal{U}}}}

\def\d{\,\mathrm{d}}
\def\diag{\mathrm{diag}}

\def\M{\mathrm{M}}

\def\Cc{C_c^{\infty}}

\def\ds{\displaystyle}
\def\bs{\backslash}
\def\t{{}^t\!}

\def\inf{\infty}
\def\fin{\mathrm{fin}}

\def\fc{{\mathfrak{c}}}
\def\fd{{\mathfrak{d}}}
\def\ff{{\mathfrak{f}}}

\def\fo{\mathfrak{o}}

\def\fO{\mathfrak{O}}
\def\fP{{\mathfrak{P}}}
\def\fC{\mathfrak{C}}

\def\cD{{\mathcal{D}}}
\def\cL{{\mathcal{L}}}

\def\cO{\mathcal{O}}
\def\cS{\mathcal S}

\def\cI{\mathcal I}
\def\cK{{\mathcal{K}}}
\def\cP{\mathcal{P}}
\def\cE{\mathcal{E}}

\def\GJ{{\mathrm{GJ}}}

\def\sD{{\mathscr{D}}}

\def\sF{\mathscr{F}}

\def\Re{{\mathop{\mathrm{Re}}}}

\def\Tr{{\mathop{\mathrm{tr}}}}

\def\Hom{{\mathop{\mathrm{Hom}}}}
\def\diag{{\mathop{\mathrm{diag}}}}

\def\Ad{{\mathop{\mathrm{Ad}}}} 
\def\vol{{\mathop{\mathrm{vol}}}}

\def\Sym{{\mathop{\mathrm{Sym}}}}

\DeclareMathOperator{\Res}{Res}

\DeclareMathOperator{\pr}{pr}

\title[Explicit mean value theorems for toric periods]{Explicit mean value theorems for toric periods and automorphic $L$-functions}

\makeatletter
\let\@wraptoccontribs\wraptoccontribs
\makeatother
\contrib[with an appendix by]{M. Suzuki, S. Wakatsuki and S. Yokoyama}

\author[Miyu Suzuki]{Miyu Suzuki} 
\address{
Depertment of Mathematics,  Faculty of Science\\
Kyoto University \\
Kitashirakawa Oiwake-cho, Sakyo-ku,  Kyoto 606-8502,   JAPAN}
\email{suzuki.miyu.4c@kyoto-u.ac.jp}

\author[Satoshi Wakatsuki]{Satoshi Wakatsuki}
\email{wakatsuk@staff.kanazawa-u.ac.jp}
\address{
Faculty of Mathematics and Physics, Institute of Science and Engineering\\
Kanazawa University \\
Kakumamachi, Kanazawa, Ishikawa, 920-1192, JAPAN}

\begin{document}


\setcounter{tocdepth}{1}

\begin{abstract}
Let $F$ be a number field and $D$ a quaternion algebra over $F$.
Take a cuspidal automorphic representation $\pi$ of $D_\A^\times$ with trivial central character and a cusp form $\phi$ in $\pi$.
Using the prehomogeneous zeta function, we find an explicit mean value of the  toric periods of $\phi$ with respect to quadratic algebras over $F$.  
The result can also be written as a mean value formula for the central values of automorphic $L$-functions twisted by quadratic characters.
\end{abstract}

\maketitle 

\tableofcontents

\section{Introduction}

In this paper, as a sequel of \cite{SW}, we prove a mean value theorem for toric periods.
Our result is based on the celebrated work of Waldspurger on toric periods and the central $L$-values.
Before stating our main theorem, we recall their result.

We fix a number field $F$ and a quaternion algebra $D$ over $F$.
Let $\A=\A_F$ be the adele ring of $F$ and $\pi=\otimes_v\pi_v$ an irreducible cuspidal automorphic representation of $D^\times_\A$ with trivial central character.
Here, $ D_\A=D\otimes_F\A$. 
Let $E$ be a quadratic \'etale algebra over $F$ embedded in $D$ and we regard $E^\times$ as a subtorus of $D^\times$.
In \cite{Wal2}, Waldspurger proved a relation between an automorphic period integral
    \[
    \cP_E(\phi)=\int_{\A_F^\times E^\times\bs\A_E^\times}\phi(h)\d^\times h,
    \hspace{25pt} \phi\in\pi
    \]
which we call a \textit{toric period} and the $L$-value $L(\tfrac12, \pi)L(\tfrac12, \pi\otimes\eta_E)$.
Here, $\eta_E=\otimes_v\eta_{E_v}$ is the quadratic character on $\A_F^\times$ attached to $E$.
To state this result, we need to introduce more notation.
Let $\langle \cdot , \cdot \rangle$ be the Petersson inner product
    \[
    \langle \phi_1, \phi_2 \rangle=\int_{\A_F^\times D^\times\bs  D_\A^\times}
    \phi_1(h)\overline{\phi_2(h)} \d h,
    \hspace{25pt} \phi_1, \phi_2\in\pi
    \]
and fix a decomposition as a product $\langle \cdot , \cdot \rangle=\prod_v\langle \cdot , \cdot \rangle_v$, where $\langle \cdot , \cdot \rangle_v$ is a $D_v^\times$-invariant inner product on $\pi_v$. 
Put 
    \[
    \alpha_{E_v}(\phi_{1, v}, \phi_{2, v})
    =\int_{F_v^\times\bs E_v^\times}\langle \pi_v(h)\phi_{1, v}, \phi_{2, v} \rangle_v
    \d h.
    \]
It converges absolutely.
Let $\zeta_F(s)$ be the Dedekind zeta function of $F$, $L(s, \pi, \Ad)$ the adjoint $L$-function of $\pi$ and $L(s, \eta_E)$ the Hecke $L$-function. 
We denote the partial Euler product of $\zeta_F(s)$, $L(s, \eta_E)$ and $L(s, \pi, \Ad)$ by $\zeta_F^S(s)$, $L^S(s, \eta_E)$ and $L^S(s, \pi, \Ad)$, respectively.

Now we are ready to state the result of Waldspurger.
Let $\phi=\otimes_v\phi_v\in\pi$ be a decomposable element and $S$ a sufficiently large finite set of places of $F$ such that $\pi_v$ is unramified for any $v\notin S$.
Waldspurger proved that there is a constant $C>0$ depending on the choice of Haar measures such that
    \[
    |\cP_E(\phi)|^2
    =C \, \frac{\zeta_F^S(2)L^S(\frac12, \pi)L^S(\frac12, \pi\otimes\eta_E)}
    {L^S(1, \pi, \Ad)L^S(1, \eta_E)}\alpha_{E, S}(\phi),
    \]
where $\alpha_{E, S}(\phi)=\prod_{v\in S}\alpha_{E_v}(\phi_v,\phi_v)$.
In particular, $L(\tfrac12, \pi)=0$ implies $\cP_E(\phi)=0$ for any $E$.
Conversely, in \cite{SW} we showed that if $L(\frac12, \pi)\neq0$ then there are infinitely many isomorphism classes of $E$ such that $\cP_E$ is not identically zero on $\pi$.
The main theorem of this paper is a refinement of this previous result.

\subsection{Main results}\label{sec:1-1}
Suppose that $S$ is sufficiently large.
Assume also that there is an $S$-tuple $\cE_S=(\cE_v)_{v\in S}$ of quadratic \'etale algebras such that $\alpha_{\cE_v}(\phi_v, \phi_v)\neq0$ and fix such $\cE_S$.
By \cite{SW}*{Theorem 1.7},  this is possible if $L(\tfrac12, \pi)\neq0$.
We fix a set of representatives $X(D)$ of isomorphism classes of quadratic \'etale $F$-subalgebras of $D$.
Let $X(D,\cE_S)$ be the set of $E\in X(D)$ such that $E_v\simeq\cE_v$ for any $v\in S$.

Let $\ff_{E_v}$ be the conductor of the quadratic character $\eta_{E_v}$, $N(\ff_{E_v})$ its norm and $N(\ff_E^S)=\prod_{v\notin S}N(\ff_{E_v})$. 
The absolute discriminant of $ F/\Q$ is denoted by $|\Delta_F|$.
If $v$ is a finite place, set $c_v=\zeta_{F_v}(1)$.
If $v$ is an archimedean place, set $c_v=1$.
Let $c_F$ be the residue of the finite part of $\zeta_F(s)$ at $s=1$.
For $v\not\in S$, set $\lambda_v=q_v^{\frac12}(\alpha_v+\alpha_v^{-1})$ where $\alpha_v\in\C^\times$ is the Satake parameter of $\pi_v$ and $q_v$ is the order of the residue field of $F_v$.
The following is a special case (the $t=1$ case) of Theorem \ref{thm:main}.

\begin{thm}[A special case of Theorem \ref{thm:main}]\label{thm:mainintro}
The limit
    \[
    \lim_{x\to\inf} \, x^{-1}                    
    \sum_E
     L(1,\eta_E)^2 \, \alpha_{E, S}(\phi)^{-1}\,  |\cP_E(\phi)|^2
    \] 
exists,  where the sum is over $E\in X(D,  \cE_S)$ such that $N(\ff_E^S)<x$.
The value of the above limit equals
    \[
    L(\tfrac12,\pi) \,
     \frac{2 \, c_F^2}{ |\Delta_F|^{\frac{1}{2}}} \,
     \prod_{v\in S}\frac{L(1, \eta_{\cE_v})}{2c_v\, L(\frac12, \pi_v)}\,
     \prod_{v\notin S}
     \left\{ 1-q_v^{-3}-\frac{q_v-1}{q_v+1}q_v^{-3}\lambda_v^2  \right\}.
    \]
\end{thm}

For a quadratic \'etale algebra $\cE'_v$ over $F_v$,  we set 
    \[
    \kappa_{\cE'_v}(\pi_v)\coloneqq\frac{\zeta_{F_v}(2)
    L(\tfrac12, \pi_v\otimes\eta_{\cE'_v})}{2c_v\, L(1, \pi_v, \Ad)}.
    \]
Then the Euler factor for $v\not\in S$ of the above mean value formula satisfies
    \[
    1-q_v^{-3}-\frac{q_v-1}{q_v+1}q_v^{-3}\lambda_v^2
    =\sum_{\cE'_v}
    \frac{\kappa_{\cE'_v}(\pi_v)}{N(\ff_{\cE'_v})}.
    \]
Here,  $\cE'_v$ runs through all quadratic \'etale algebras over $F_v$.
We do not yet have conceptual understanding of this expression. 

By using Waldspurger's formula for $|\cP_E(\phi)|^2$, we can rewrite this result as a mean value theorem for the special values of the twisted $L$-functions.

\begin{thm}[A special case of Corollary \ref{cor:mvfL}]\label{thm:mainintro2}
Let $S$ be an arbitrary finite set of places of $F$.
We allow the possibility that $S$ is empty.
Suppose that $L(\tfrac12, \pi)\neq0$.
The limit
    \[
    \lim_{x\to\inf} \, x^{-1} \sum_E
    L(\tfrac12, \pi\otimes\eta_E)
    \] 
exists,  where the sum is over $E\in X(D,  \cE_S)$ such that $N(\ff_E^S)<x$.
The value of the above limit equals
    \[
     2 \, c_F^2 |\Delta_F|^\frac12 \,\frac{L(1, \pi, \Ad)}{\zeta_F(2)}
     \prod_{v\in S}\kappa_{\cE_v}(\pi_v)\cdot
      \prod_{v\notin S}\left(\sum_{\cE'_v}
     \frac{\kappa_{\cE'_v}(\pi_v)}{N(\ff_{\cE'_v})}\right).
    \]
\end{thm}

As corollaries of these results, we record two mean value theorems in the following special cases:
\begin{itemize}
\item[(i)] The case where $\phi$ corresponds to a Hecke eigenform $f\in S_k(\SL_2(\Z))$.
In this case, $F=\Q$ and $D$ is the split quaternion algebra.

\item[(ii)] The case where $F=\Q$ and $D$ is a quaternion division algebra of odd prime discriminant.
\end{itemize}
Appendix \ref{sec:appendix} by Shun'ichi Yokoyama and the authors provides numerical examples for these two special cases using \texttt{Magma} \cite{BCP}.

We describe one consequence of the explicit mean value theorem for the case (i).
Let $k$ be an even positive integer and $f\in S_k(\SL_2(\Z))$ a normalized Hecke eigenform.
We denote by $a_n$ the $n$-th Fourier coefficient of $f$ and by $\langle\cdot,\cdot\rangle$ the Petersson inner product.
Let $\sigma=\otimes_v \sigma_v$ be the cuspidal automorphic representation of $\GL_2(\A_\Q)$ generated by $f$.
Theorem \ref{thm:mainintro2} in this case reads as follows.

\begin{thm}[Corollary \ref{cor:mvfhol}]
The limit
    \[
    \lim_{x\to\inf} x^{-1}
    \sum_E  L(\tfrac{1}{2},\sigma\otimes\eta_E) 
    \]
exists, where the sum is over real quadratic fields $E$ with $|\Delta_E|<x$.
The value of the above limit equals  
    \[
    \frac{6(8\pi)^\frac k2\Gamma(\frac k2)}{\pi\, \Gamma(k)} \langle f, f \rangle
    \prod_{p}\left\{1-p^{-3}-\frac{p-1}{p+1}p^{-k-1}a_p^2 \right\},
    \]
where the product is over all primes.
\end{thm}

\subsection{Methods and related works}\label{sec:1-2}
We prove Theorem \ref{thm:mainintro} by applying Tauberian theorem to the prehomogeneous zeta function with toric periods which is introduced in \cite{SW}.
As we see in \cite{SW}*{Theorem 4.10} (or Theorem \ref{thm:Dirichlet} of this paper), our zeta function is roughly a counting function of $L(1, \eta_E)^2\alpha_{E, S}(\phi)^{-1}|\cP_E(\phi)|^2$.
Based on this interpretation, we follow the general strategy toward the density theorems discussed by Wright and Yukie in the introduction to \cite{WY}.
Since the global zeta function is an approximation of the counting function, we need to carry out the filtering process developed by Datskovsky and Wright \cite{DW}.
The main step of the filtering process is to find a uniform estimate for the contribution of the local zeta functions to the residue of the global zeta function.
We emphasize that we use the bound of Blomer and Brumley \cite{BB} for Hecke eigenvalues to obtain this uniform estimate. 

When $F=\Q$, similar results are obtained by many researchers.
Here, we mention only a few of them.
Radziwi\l\l\, and Soundararajan \cite{RS} studied the mean values for quadratic twists of Hasse-Weil $L$-functions associated with elliptic curves.
Soundararajan and Young \cite{SY} obtained second moments for quadratic twists of modular $L$-functions. 
Higher moments for products of automorphic $L$-functions are studied by Sono \cite{Sono}.
These results concern estimates of the order and error terms.
There are several researches which aimed to determine the leading term.
See \cite{GV} and \cite{KZ}. 
In \cite{CFKRS}, a conjectural formula for the leading term was proposed.
Their formula looks quite different from ours and we did not confirm that our formula coincides with theirs.

For a general number field, several results are obtained by analyzing a multiple Dirichlet series. 
Friedberg and Hoffstein \cite{FH} proved non-vanishing of central $L$-values for infinitely many quadratic twists.   
Their results also imply the existence of the density and its positivity.
Bump, Friedberg and Hoffstein obtained an explicit mean value formula with an error term estimate for quadratic twists of $L$-functions for $\GL_3(\A_\Q)$ times certain correction terms \cite{BFH}*{Theorem 3.8}.
One may be able to do the same for quadratic twists of $L$-functions for $\GL_2(\A_F)$ for a general number field $F$ using the technique developed in \cite{FF}, \cite{BFH} and \cite{FHL}.
However, it is not clear whether one can recover Theorem \ref{thm:mainintro} by this approach since their sum involves correction terms which contain information about ramification of the automorphic representation.
An important step of the proof of Theorem \ref{thm:mainintro} is to show that the ramification of the automorphic representation is reflected in the mean value as a local period.

Compared with others, the approach using prehomogeneous zeta functions has several advantages:
\begin{itemize}
\item[(1)] One can compute the mean value explicitly and the resulting formula is so simple that it can be numerically checked.
\item[(2)] One can consider contributions of quadratic extensions with given local behaviors at arbitrarily chosen finite number of places.
\item[(3)] The mean value formula has an Euler product, which visualizes the contribution of each local component of the automorphic representation.
\end{itemize}

A prototype of our zeta function was first introduced by Sato \cite{Sato} and its local theory is extensively studied by Wen-Wei Li (\cite{Li1}, \cite{Li2}, \cite{Li3} and \cite{Li4}) in a quite general setting.
The results of this paper and the previous one \cite{SW} are first applications to the non-vanishing of automorphic periods and $L$-functions.

\subsection{Organization of the paper}\label{sec:1-3}
The organization of the paper is as follows.
In Section \ref{sec:mvf}, after introducing notations we use throughout the paper, we state our main theorem and deduce from it the above mentioned two special cases.
In Section \ref{sec:pvzeta}, we briefly recall the definition and necessary properties of the prehomogeneous zeta function with toric periods from \cite{SW}.
The global zeta function can be written as a Dirichlet series with appropriate weight factors.
We show that this Dirchlet series has a simple pole at $s=\tfrac32$ unless it is identically zero and determine the residue. 
The residue has an Euler product and we see that each local factor is essentially the local period $\alpha_{\cE_v}(\phi_v, \phi_v)$ which appears in Waldspurger's formula.
In Section \ref{sec:proof}, we apply the filtering process and prove the main theorem.
Finally in Appendix \ref{sec:appendix}, we provide numerical examples for the mean values of the special values of the twisted $L$-functions of the elliptic modular cusp form of level $1$, weight $12$ and the toric periods of an algebraic modular form.

\vspace{2mm}
\noindent\textbf{Acknowledgments.} 
The authors would like to thank Takashi Taniguchi, Masao Tsuzuki, Solomon Friedberg and Jeffrey Hoffstein for helpful comments.
The authors are grateful to the anonymous referees for helpful comments that improved the previous version of this article.
M.S. was partially supported by JSPS Grant-in-Aid for Early-Career Scientists No.22K13891.
S.W. was partially supported by JSPS Grant-in-Aid for Scientific Research (C) No.20K03565, (B) No.21H00972.
S.Y. was partially supported by JSPS Grant-in-Aid for
Scientific Research (C) No.20K03537.


\section{Mean value theorems}\label{sec:mvf}


\subsection{Preliminaries}


\subsubsection{Number fields}\label{sec:n2}

Let $F$ be a number field, $\fo_F$ its ring of integers and $\Sigma$ the set of all places of $F$. 
For $v\in\Sigma$, let $F_v$ denote the completion of $F$ at $v$.
In particular, when $F=\Q$ and $v=\infty$ is its unique infinite place, $\Q_\infty$ denotes the real number field $\R$.
If $v$ is a finite place, we write $\fo_v=\fo_{F,v}$ for the ring of integers of $F_v$, $\varpi_v$ for a prime element of $\fo_v$ and $q_v$ for the order of the residue field $\fo_v/\varpi_v\fo_v$.

Let $\A=\A_F$ denote the adele ring of $F$ and $\A_\fin=\A_{F, \fin}$ its finite part.
Set $F_\inf=\prod_{v|\inf}F_v$, where the product runs over the set of infinite places of $F$.
We write by $r_1$ (resp. $r_2$) the number of real (resp. complex) places of $F$.

For an algebra $E$ over $F$ and $v\in\Sigma$, we write $E_v=E\otimes_FF_v$.
Suppose that $E$ is a quadratic \'etale algebra over $F$.
We write the maximal order of $E$ (resp.\,$E_v$) by $\fo_E$ (resp.\,$\fo_{E,  v}$).
Let $\eta=\eta_E$ be the corresponding quadratic character on $\A^\times$.
For $v\in\Sigma$, let $\eta_v=\eta_{E_v}$ be the quadratic character on $F_v^\times$ corresponding to $E_v$.

Let $\d x$ be the Haar measure on $\A$ normalized so that $\vol(\A/F)=1$. 
For $v\in\Sigma$, we fix a Haar measure $\d x_v$ on $F_v$ as follows.
If $v$ is finite, $\d x_v$ is normalized so that $\vol(\fo_v)=1$. 
If $v$ is a real place,  let $\d x_v$ denote the ordinary Lebesgue measure on $\R$. 
For a complex place $v$, we set $\d x_v=2 \d x_{1,v} \d x_{2,v}$,  where $x_v=x_{1,v}+x_{2,v}\sqrt{-1}\in \C$ and $\d x_{i,v}$,  $i=1,  2$ are the Lebesgue measures on $\R$. 
Then, we obtain $\d x = |\Delta_F|^{-1/2} \prod_{v\in\Sigma} \d x_v$, where $|\Delta_F|$ is the absolute discriminant of $F/\Q$.

Let $|\cdot|_v$ denote the normalized absolute value of $F_v$ and $|\cdot|$ the idele norm on $\A^\times$.
Take a non-trivial additive character $\psi_\Q$ on $\A_\Q/\Q$, and set $\psi_F=\psi_\Q \circ \Tr_{F/\Q}$. 
Then, $\d x$ is the self-dual Haar measure with respect to $\psi_F$. For each $v\in \Sigma$, we write the restriction of $\psi_{F}$ to $F_v$ by $\psi_{F_v}$.
For each place $v\in\Sigma$, let $\d^\times x_v$ denote the Haar measure $\d^\times x_v=c_v\, |x|_v^{-1}\d x_v$ on $F_v^\times$, where
    \[
    c_v=
    \begin{cases} 
    1                            & \text{if $v$ is an infinite place}, \\ 
    (1-q_v^{-1})^{-1} & \text{if $v$ is a finite place}. 
    \end{cases}
    \]
Define the Haar measure on $\A^\times$ by $\d^\times x=|\Delta_F|^{-1/2}\prod_{v\in\Sigma}\d^\times x_v$, and normalize the Haar measure $\d^1 y$ on $\A^1=\{x\in\A^\times \mid |x|=1\}$ so that $\d^\times x$ equals $\d^1 y\times\d t/t$ under the isomorphism $\A^1\times\R_{>0}\xrightarrow{\sim}\A^\times$,  $(y,  t)\mapsto yt$.
Here,  $\d t$ is the Lebesgue measure on~$\R$ and $\d t/t$ is the Haar measure on $\R_{>0}$.

Let $\zeta_F(s)$ (resp. $\zeta_{F,\fin}(s)$) denote the completed (resp. the finite part of) Dedekind zeta function of $F$.
For a character $\chi$ on $F^\times\R_{>0}\bs \A^\times$, write $L(s,\chi)$ (resp. $L_\fin(s,\chi)$) for the completed (resp. the finite part of) $L$-function of the Hecke character $\chi$.
When $\chi=\mathbf{1}$ is the trivial character, $L_\fin(s,\mathbf{1})=\zeta_{F,\fin}(s)$ has a simple pole at $s=1$ with residue $c_F\coloneqq\vol(F^\times\bs \A^1)$.
By abuse of notation, we write $L_\fin(1, \mathbf{1})=c_F$ so that $L_\fin(1,  \chi)$ is defined for any Hecke character $\chi$.


\subsubsection{Quaternion algebras}\label{sec:n1}

Let $F$ be a field of characteristic zero and $D$ a quaternion algebra over $F$.
When $D$ is not division, we identify $D$ with $\M_2(F)$ so that the main involution $\iota$ of $D$ is given by 
   $x^\iota=
    \begin{psmallmatrix} 
    0  & 1 \\ 
    -1 & 0 
    \end{psmallmatrix}\,{}^t\!x 
    \begin{psmallmatrix} 
    0 & -1 \\ 
    1 & 0 
    \end{psmallmatrix}$. 
When $D$ is division and $E$ is its subfield which is a quadratic \'etale algebra over $F$, we can find an element $b\in F^\times$ so that $D$ is isomorphic to the matrix algebra
    \[
    \left\{  
    \begin{pmatrix} 
    \xi                        & \eta \\ 
    b\overline{\eta} & \overline{\xi} 
    \end{pmatrix} \in M_2(E) \mid  \xi,\eta\in E \right\},
    \]
where $\bar{\cdot}$ is the unique non-trivial $F$-algebra automorphism of $E$.
The main involution $\iota$ on $D$ is given by 
    $\left(
    \begin{smallmatrix} 
    \xi                        & \eta \\ 
    b\overline{\eta} & \overline{\xi} 
    \end{smallmatrix}\right)^\iota=\left(
    \begin{smallmatrix} 
    \overline{\xi}        & -\eta \\ 
    -b\overline{\eta} & \xi  
    \end{smallmatrix}\right)$.
We write the reduced norm and the reduced trace by $\det(x)= x\, x^\iota\in F$  and $\Tr(x)=x+x^\iota\in F$ for $x\in D$, respectively.

From the Skolem-Noether theorem, two quadratic \'etale $F$-subalgebras of $D$ are isomorphic if and only if they are conjugate to each other. 
Let $X(D)$ denote a set of representatives of isomorphism classes of quadratic \'etale $F$-subalgebras of $D$.
For each $\cE\in X(D)$, choose an element $\delta_\cE \in D$ so that $\cE=F+F\delta_\cE$ and $\Tr(\delta_\cE)=0$.
Set $d_\cE=\delta_\cE^2\in F$. 
Note that we have $\det(\delta_\cE)=-d_\cE$.

Suppose $F$ is a number field and $E\in X(D)$.
For a place $v\in\Sigma$, $D_v=D\otimes_FF_v$ is a quaternion algebra over $F_v$ and $E_v=E\otimes_FF_v$ is an $F_v$-subalgebra of $D_v$.
Unless otherwise mentioned, we assume $\delta_{E_v}=\delta_E$ and $d_{E_v}=d_E$ under the natural embeddings $ E\hookrightarrow E_v$ and $D\hookrightarrow D_v$.


\subsubsection{Measures on algebraic groups}\label{sec:measure1}

From now on, let $F$ be a number field.
For an algebraic group $G$ over $F$, we write $G(F_\infty)$, $G(\A)$ and $G(\A_\fin)$ by $G_\infty$, $G_\A$ and $G_{\A_\fin}$, respectively.
Let $K=\prod_v K_v$ be a maximal compact subgroup  of $\GL_2(\A)$,  where
    \[
    K_v=
        \begin{cases}
        \GL_2(\fo_v) & \text{ if $v$ is finite,} \\
        \O(2) & \text{ if $v$ is real,} \\
        \U(2) & \text{ if $v$ is complex}.
        \end{cases}
    \]
Let $\d k_v$ be the Haar measure on $K_v$ normalized so that $\vol(K_v)=1$ and define the Haar measure $\d k$ on $K$ by $\d k=\prod_v\d k_v$.

Let $D$ be a quaternion algebra over $F$ and $Z$ the center of $D^\times$ and $\d z$ the Haar measure on $Z(\A)=\A^\times$ given by $\d z=\frac{\d t}{t} \, \d^1 a$, where $z=t^{\frac12}a I_2$ with $t\in\R_{>0}$ and $a\in\A^1$.
Note that $\d z=2\d^\times x$, where $d^\times x$ is the measure on $\A^\times$ chosen in \S\,\ref{sec:n2}.
Let $\d g$ be the Tamagawa measure on $D_\A^\times$. 
The algebraic group $D^\times/ F^\times$ over $F$ will be denoted by $PD^\times$.
Then we have $\vol(PD^\times\bs PD^\times_\A)=c_F^{-1}$ with respect to the quotient measure $\frac{\d g}{\d z}$.
Let $\langle \cdot,\cdot\rangle$ be the inner product on $L^2(PD^\times\bs PD^\times_\A)$ given by
    \begin{equation}\label{eq:innprod}
    \langle \varphi,\varphi'\rangle=\int_{PD^\times\bs PD^\times_\A} 
    \varphi(g)\, \overline{\varphi'(g)} \frac{\d g}{\d z}, 
    \hspace{25pt} \varphi, \varphi'\in L^2(PD^\times\bs PD^\times_\A).
    \end{equation}

Let $\d x$ denote the Tamagawa measure on $D$ and $\d x_v$ the local Tamagawa measure on $D_v=D\otimes F_v$ for $v\in\Sigma$.
We have
    \begin{equation}\label{eq:tama}
    \d x= |\Delta_F|^{-2}\prod_{v\in \Sigma} \d x_v.
    \end{equation}
For a quadratic \'etale algebra $\cE_v$ over $F_v$,  we write an element $h_{\cE_v}\in\cE_v$ in the form $h_{\cE_v}=a_v+\delta_{\cE_v}b_v$ with $a_v,  b_v\in F_v$.
Define a Haar measure $\d h_{\cE_v}$ on $\cE_v^\times$ by
    \begin{equation}\label{eq:localmeah}
    \d h_{\cE_v}= c_v\, L(1,\eta_{\cE_v}) \, 
    \frac{\d a_v \, \d b_v}{|a_v^2-d_{\cE_v} b_v^2|_v},        
    \end{equation}
where $\eta_{\cE_v}$ is the quadratic character on $F^\times_v$ corresponding to $\cE_v$ and $\d a_v$,  $\d b_v$ are the Haar measures on $F_v$ chosen in \S\,\ref{sec:n2}.
For a quadratic \'etale algebra $E$ over $F$, the Tamagawa measure $\d h_E$ on $(\A_F\otimes_FE)^\times$ is given by
    \[
    \d h_E=\frac{1}{c_F\, L(1,\eta_E)\, |\Delta_F|} \prod_v \d h_{E_v}.
    \]
Here, $\eta_E=\otimes_{v\in\Sigma} \eta_{E_v}$ denotes the quadratic character on $\A_F^\times$ corresponding to $E$.


\subsubsection{Toric periods}\label{sec:w0504}

Throughout this paper, we assume that $\pi=\otimes_{v\in\Sigma}\pi_v$ is an irreducible cuapidal automorphic representation of $PD^\times_\A$ which is not $1$-dimensional.
For each place $v\in\Sigma$, take a $D_v^\times$-invariant non-degenerate Hermitian pairing $\langle\cdot,\cdot\rangle_v$ on $\pi_v$ so that we have
    \[
    \langle\cdot,\cdot\rangle=\prod_v\langle\cdot,\cdot\rangle_v.
    \]
For $\cE_v\in X(D_v)$,  let $\d h_v=\frac{\d h_{\cE_v}}{\d^\times z_v}$ be the quotient measure on $F_v^\times\bs\cE_v^\times$.  
The integral
    \[
    \alpha_{\cE_v}(\varphi_v,\varphi'_v)=\int_{F_v^\times\bs \cE_v^\times}
    \langle\pi_v(h_v)\varphi_v, \varphi'_v\rangle_v\,\d h_v, 
    \hspace{10pt} \varphi_v, \varphi'_v\in\pi_v
    \]
converges absolutely and defines an element of $\Hom_{\overline{\cE_v^\times}\times \overline{\cE_v^\times}}(\pi_v\boxtimes\bar{\pi}_v, \C)$, 
Here we write $\overline{A}$ for the image of $A(\subset D_v)$ under the projection $D_v^\times\to PD_v^\times$.
Let $E\in X(D)$.
We define the Haar measure $\d h$ on $\A_F^\times E^\times\bs \A_E^\times$ by $\d h= (c_F^{-1} \d^\times z)\bs \d h_E$.
A linear form $\cP_E$ on $\pi$ is defined by 
    \[
    \cP_E(\varphi)=\int_{\A_F^\times E^\times\bs \A_E^\times}\varphi(h) \, \d h , 
    \hspace{10pt} \varphi \in \pi.
    \]
We say that $\pi$ is $E^\times$-distinguished if $\cP_E$ is not identically zero.
The linear form $\cP_E$ is called a toric period.
Similarly for $\cE_v\in X(D_v)$, we say that $\pi_v$ is $\cE_v^\times$-distinguished if $\Hom_{\overline{\cE_v^\times}}(\pi_v, \C)\neq0$.

Let $\phi=\otimes_{v\in\Sigma}\phi_v$ be a decomposable element of $\pi$.
In \cite{Wal2}, Waldspurger proved a formula which relates the toric period and the special values of automorphic $L$-functions: 
    \begin{equation}\label{eq:Wald}
    |\cP_E(\phi)|^2=\frac{1}{|\Delta_F|} \,
    \frac{\zeta_F(2)L(\frac12, \pi)L(\frac12, \pi\otimes\eta_E)}
    {L(1, \pi, \Ad)L(1, \eta_E)^2} \,
    \prod_v\alpha_{E_v}^\#(\phi_v, \phi_v).
    \end{equation}
Here, 
    \[
    \alpha_{E_v}^\#(\phi_v, \phi_v)=
    \frac{L(1, \pi_v, \Ad)L(1, \eta_{E_v})}
    {\zeta_{F, v}(2)L(\frac12, \pi_v)L(\frac12, \pi_v\otimes\eta_{E_v})} \,
    \alpha_{E_v}(\phi_v, \phi_v)
    \]
is the normalized local period.
Note that the choice of our Haar measure on $PD_\A^\times$ is $2c_F$ times the Tamagawa measure.
See also \cite{II}*{Section 6}.

Let $S$ be a finite subset of $\Sigma$ satisfying the following conditions.

\begin{condition}\label{condition}
For $v\not\in S$,  
\begin{itemize}
\item $v$ is a finite place which is not dyadic.
\item $D_v$ is split,  in particular $PD_v^\times\simeq\PGL_2(F_v)$.
\item $\pi_v$ is unramified and $\phi_v$ is the spherical vector, which is normalized so that $\langle \phi_v,\phi_v \rangle_v=1$. 
\item let $K_v$ be the maximal compact subgroup corresponding to $\PGL_2(\fo_v)$ under a fixed isomorphism $PD_v^\times\simeq\PGL_2(F_v)$.
\item for every $E\in X(D)$, we have $d_{E_v}\in\fo_v\setminus\varpi_v^2\fo_v$, and the maximal compact subgroup of $E_v^\times$ is contained in $K_v$ (that is, $\overline{\fo_{E,  v}^\times}\subset K_v$).
\end{itemize}
\end{condition}
Since we have $PD_\A^\times=PD^\times \prod_{v\in S}PD_v^\times \prod_{v\notin S}K_v$ for sufficiently large $S$, by taking a suitable $PD^\times$-conjugate, we may assume that our fixed embedding $E\hookrightarrow D$ satisfies the last condition of Condition \ref{condition}.

Since $\alpha_{E_v}^\#(\phi_v,  \phi_v)= 1$ for any $v\notin S$ (see Corollaries \ref{cor:fP} and \ref{cor:fPun}), the formula \eqref{eq:Wald} is equivalent to
    \begin{equation}\label{eq:Wald2}
    |\cP_E(\phi)|^2=\frac{1}{|\Delta_F|} \,
    \frac{\zeta_F^S(2) \, L^S(\frac12, \pi) \, L^S(\frac12, \pi\otimes\eta_E)}
    {L^S(1, \pi, \Ad)\, L^S(1, \eta_E)\, L(1,\eta_E)} \,
    \alpha_{E,  S}(\phi),
    \end{equation}
where $\alpha_{E, S}(\phi)=\prod_{v\in S}\alpha_{E_v}(\phi_v,\phi_v)$.
Note that if both sides of \eqref{eq:Wald2} is not equal to zero,  then $\alpha_{E,  S}(\phi)^{-1}|\cP_E(\phi)|^2$ depends only on the isomorphism class of $E$ and is independent of its realization as a subalgebra of $D$.

Take an element $\cE_S=(\cE_v)_{v\in S}\in X(D_S)$.  
Let $E\in X(D)$ and suppose that $E_v$ is isomorphic to $\cE_v$ for each $v\in S$. 
In the case $\alpha_{E, S}(\phi)\neq0$, $\alpha_{E, S}(\phi)^{-1}|\cP_E(\phi)|^2$ is explicitly given by \eqref{eq:Wald2}.  
In the case $\alpha_{E, S}(\phi)=0$, by abuse of notation, from the viewpoint of \eqref{eq:Wald2} we formally define 
    \[
     \alpha_{E, S}(\phi)^{-1}|\cP_E(\phi)|^2\coloneqq \frac{1}{|\Delta_F|} \,
    \frac{\zeta_F^S(2) \, L^S(\frac12, \pi) \, L^S(\frac12, \pi\otimes\eta_E)}
    {L^S(1, \pi, \Ad)\, L^S(1, \eta_E)\, L(1,\eta_E)}  
    \] 
    when $\pi$ is $E^\times$-distinguished, and
    \[
     \alpha_{E, S}(\phi)^{-1}|\cP_E(\phi)|^2\coloneqq 0 \qquad \text{when $\pi$ is not $E^\times$-distinguished.}
    \]
Note that $\alpha_{E, S}(\phi)=0$ implies $\cP_E(\phi)=0$ even if $\pi$ is $E^\times$-distinguished.
In addition, by abuse of notation, we define
    \begin{equation}\label{eq:alphaEv}
         \alpha_{E}^{\cE_S}(\phi) |\cP_E(\phi)|^2 \coloneqq \alpha_{E,S}(\phi)^{-1}|\cP_E(\phi)|^2\times \alpha_{\cE_S}(\phi,\phi) \times \prod_{v\in S}\left|\frac{d_{\cE_v}}{d_{E_v}}\right|_v^{\frac12}
    \end{equation}
where $\alpha_{\cE_S}(\phi,\phi)\coloneqq\prod_{v\in S}\alpha_{\cE_v}(\phi_v,\phi_v)$. 
When $\alpha_{E, S}(\phi)\neq0$, we have
\[
\alpha_{E}^{\cE_S}(\phi) = \frac{\alpha_{\cE_S}(\phi,\phi)}{\alpha_{E,S}(\phi)} \prod_{v\in S}\left|\frac{d_{\cE_v}}{d_{E_v}}\right|_v^{\frac12}. 
\]
Note that $\alpha_{E}^{\cE_S}(\phi)|\cP_E(\phi)|^2$ is non-negative (see Lemma \ref{lem:non-negative}).

\subsection{Mean value theorem for toric periods}\label{sec:mvfgeneral}

Let $\phi=\otimes_v\phi_v$ be a decomposable element in $\pi$. 
We take a finite subset $S$ of $\Sigma$  and a set $X(D)$ of representatives of isomorphism classes of quadratic \'etale $F$-subalgebras of $D$ so that $S$ satisfies Condition \ref{condition} for $\phi$ and all (but finitely many) $E\in X(D)$.

Put $X(D_S)\coloneqq\prod_{v\in S}X(D_v)$. 
We assume that there exists $\cE_S=(\cE_v)_{v\in S}\in X(D_S)$ such that $\alpha_{\cE_v}(\phi_v, \phi_v)\neq0$ for any $v\in S$.
In \cite{SW}*{Theorem 1.7}, we showed that this is the case if $L(\tfrac12, \pi)\neq0$.
Fix such $\cE_S$.
Let $X(D,\cE_S)$ be the set of $E\in X(D)$ such that $E_v\simeq\cE_v$ for each $v\in S$.

\begin{thm}\label{thm:main}
Let $\phi=\otimes_v\phi_v$,  $S\subset\Sigma$ and $\cE_S\in X(D_S)$ be as above.
In particular,  we assume that Condition \ref{condition} is satisfied.
For $t>0$, we have
    \begin{align}\label{eq:mvf}
    & \lim_{x\to\inf} \, x^{-t}                    
    \sum_{\substack{E\in X(D,\cE_S),\\ N(\ff_E^S)<x}} 
     N(\ff_E^S)^{t-1} \, L(1,\eta_E)^2 \, \alpha_{E, S}(\phi)^{-1}|\cP_E(\phi)|^2 \\ 
    &\hspace{50pt} = L(\tfrac12,\pi) \,
     \frac{2 \, c_F^2}{ t \,  |\Delta_F|^{\frac{1}{2}}} \,
     \prod_{v\in S}\frac{L(1, \eta_{\cE_v})}{2c_v\, L(\frac12, \pi_v)}\cdot
     \prod_{v\notin S}
     \left\{ 1-q_v^{-3}-\frac{q_v-1}{q_v+1}q_v^{-3}\lambda_v^2  \right\}.
      \nonumber
    \end{align}
Here, $\ff_{E_v}$ is the conductor of the quadratic character $\eta_{E_v}$, $N(\ff_E^S)=\prod_{v\notin S}N(\ff_{E_v})$ and $\lambda_v=q_v^{\frac12}(\alpha_v+\alpha_v^{-1})$,   where $\alpha_v\in\C^\times$ is the Satake parameter of $\pi_v$. 
Note that the mean value on the left-hand side depends only on the equivalence class of $\cE_S$ and independent of the choice of the representative element $\cE_S$. 
\end{thm}
\begin{rem}\label{rem:mvf}
    The factor $\alpha_{E, S}(\phi)^{-1}|\cP_E(\phi)|^2$ is independent of the choice of $\otimes_{v\in S}\phi_v$, and hence the both sides of \eqref{eq:mvf} are also so.  
\end{rem}
\begin{rem}
For any $v\in\Sigma$ and a quadratic \'etale algebra $\cE_v'$ over $F_v$,  we set 
    \[
    \kappa_{\cE'_v}(\pi_v)\coloneqq\frac{\zeta_{F_v}(2)
    L(\tfrac12, \pi_v\otimes\eta_{\cE'_v})}{2c_v\, L(1, \pi_v, \Ad)}.
    \]
A direct computation shows that the Euler factor for $v\not\in S$ of \eqref{eq:mvf} satisfies 
    \[
    1-q_v^{-3}-\frac{q_v-1}{q_v+1}q_v^{-3}\lambda_v^2
    =\sum_{\cE'_v}\frac{\kappa_{\cE'_v}(\pi_v)}{N(\ff_{\cE'_v})},
    \]
where $\cE'_v$ runs through all quadratic \'etale algebras.
\end{rem}

Substituting Waldspurger's formula \eqref{eq:Wald2},  one can rewrite \eqref{eq:mvf} as a mean value formula for $L$-values.
\begin{cor}\label{cor:mvfL}
We keep the notation as in Theorem \ref{thm:main}.
Suppose that $L(\tfrac12, \pi)\neq0$.
For $t>0$, we have
    \begin{align}\label{eq:mvfL}
    &\lim_{x\to\infty}x^{-t}\sum_{\substack{E\in X(D, \cE_S)\\ N(\ff_E^S)<x}}
    N(\ff_E^S)^{t-1}L(\tfrac12, \pi\otimes\eta_E) \\
      &\hspace{50pt}
      =\frac{2\, c_F^2|\Delta_F|^\frac12}{t} \, \frac{L(1, \pi, \Ad)}{\zeta_F(2)}\, 
      \prod_{v\in S}\kappa_{\cE_v}(\pi_v)\cdot
      \prod_{v\notin S}\left(\sum_{\cE'_v\in X(D_v)}
     \frac{\kappa_{\cE'_v}(\pi_v)}{N(\ff_{\cE'_v})}\right).
     \nonumber
     \end{align}
\end{cor}

\begin{rem}\label{rem:dependence}
We examine the dependence on $S$ of the mean value formula \eqref{eq:mvfL} for $L$-values.
Suppose that $T$ is a finite subset of $\Sigma$ such that $S\subset T$.
Note that $T$ satisfies Condition \ref{condition} if $S$ does.
Recall that $\cE_S\in X(D_S)$ is fixed.
It is easily seen that
    \begin{align*}
    &\sum_{\substack{E\in X(D, \cE_S)\\ N(\ff_E^S)<x}}
    N(\ff_E^S)^{t-1}L(\tfrac12, \pi\otimes\eta_E) \\
    &=\sum_{\substack{\cL_T\in X(D_T) \\ \cE_S\subset\cL_T}}
    \left(\prod_{v\in T\setminus S}N(\ff_{\cL_v})\right)^{t-1}
    \sum_{\substack{E\in X(D, \cL_T)\\ 
    N(\ff_E^T)<\left(\prod_{v\in T\setminus S}N(\ff_{\cL_v})\right)^{-1}x}}
    N(\ff_E^T)^{t-1}L(\tfrac12, \pi\otimes\eta_E),
    \end{align*}
where $\cL_T=(\cL_v)_{v\in T}$ runs over elements in $X(D_T)$ such that $\cE_v=\cL_v$ for each $v\in S$.
Hence the left hand side of \eqref{eq:mvfL} becomes
    \[
    \sum_{\substack{\cL_T\in X(D_T) \\ \cE_S\subset\cL_T}}
    \left(\prod_{v\in T\setminus S}N(\ff_{\cL_v})\right)^{-1} \cdot
    \lim_{x\to\infty}x^{-t} \sum_{\substack{E\in X(D, \cL_T)\\ N(\ff_E^T)<x}}
    N(\ff_E^T)^{t-1}L(\tfrac12, \pi\otimes\eta_E).
    \]
Also,  we obviously have
    \begin{multline*}
    \prod_{v\in S}\kappa_{\cE_v}(\pi_v)\cdot
      \prod_{v\notin S}\left(\sum_{\cE'_v\in X(D_v)}
     \frac{\kappa_{\cE'_v}(\pi_v)}{N(\ff_{\cE'_v})}\right) \\
     =\sum_{\substack{\cL_T\in X(D_T) \\ \cE_S\subset\cL_T}}
     \left(\prod_{v\in T\setminus S}N(\ff_{\cL_v})\right)^{-1} \cdot
     \prod_{v\in T}\kappa_{\cL_v}(\pi_v)\cdot
      \prod_{v\notin T}\left(\sum_{\cL'_v\in X(D_v)}
     \frac{\kappa_{\cL'_v}(\pi_v)}{N(\ff_{\cL'_v})}\right).
    \end{multline*}
It follows that the mean value formula  \eqref{eq:mvfL} holds once we prove it for some finite set $T$ containing $S$.
This means that in order to prove Corollary \ref{cor:mvfL} (and equivalently,  to prove Theorem \ref{thm:main}),  we may assume that $S$ is sufficiently large.
Moreover,  Corollary \ref{cor:mvfL} holds for arbitrary finite set $S$ and in particular we can remove Condition \ref{condition}. 
Note,  however,  that this is not the case in Theorem \ref{thm:main} since we have the equivalence between Theorem \ref{thm:main} and Corollary \ref{cor:mvfL} only when $S$ satisfies Condition \ref{condition}.
\end{rem}


\subsection{Mean value theorem for elliptic modular forms}\label{sec:mvfhol}

In this subsection, we specialize Corollary \ref{cor:mvfL} to the case of elliptic modular forms of level $1$. 
Let $k$ be a non-negative even integer and $f\in S_{k}(\SL_2(\Z))$ be a weight $k$ normalized cuspidal Hecke eigenform.
Let $ f(z)=\sum_{n=1}^\infty a_nq^n$ be the Fourier expansion, where $q=e^{2\pi\sqrt{-1}z}$.
Let $\sigma=\otimes_v\sigma_v$ be the cuspidal automorphic representation of $\GL_2(\A_\Q)$ corresponding to $f$.
Note that $\lambda_p=p^{1-\frac{k}{2}}a_p$ for each prime $p$.

Note that $\sigma$ does not have non-trivial toric periods with respect to any imaginary quadratic fields since $\sigma_\infty$ does not have  non-zero $\SO(2)$-fixed vector.
Therefore we only consider toric periods with respect to real quadratic fields.
We can apply Corollary \ref{cor:mvfL} to this case with $F=\Q$,  $D=\M_2(\Q)$ and $S=\{2, \infty\}$.

For a quadratic field $E$,  recall that we fixed $d_E\in\Q$ such that $E=\Q(\sqrt{d_E})$ (see \S\,\ref{sec:n1}).
We further assume that $d_E$ is square-free and set 
    \begin{equation}\label{eq:a_E}
    a_E\coloneqq
        \begin{cases}
        |d_E| & \text{if $d_E\equiv 1,  3 \bmod 4$,} \\
        \frac{|d_E|}{2} & \text{otherwise.}
        \end{cases}
    \end{equation}
Note that when $S=\{2,  \infty\}$,  we have $N(\ff_E^S)=\prod_{v\notin S}|d_{E_v}|_v^{-1}=\prod_{v\in S}|d_E|_v=a_E$.

\begin{thm}\label{thm:mvfhol}
Set $S=\{2, \infty\}$ and take a pair of quadratic \'etale algebras $\cE_S=(\cE_v)_{v\in S}$ so that $\cE_\infty=\R\times\R$.
For a positive real number $t>0$,
    \begin{align}\label{eq:mvfhol1}
    &\hspace{-15pt} \lim_{x\to\inf}\, x^{-t} 
    \sum_{\substack{ E\in X(\M_2(\Q),\cE_S),\\ a_E<x}}
      a_E^{t-1} \,  L(\tfrac12, \sigma\otimes\eta_E)                                \nonumber \\ 
    & = \frac{2}{\pi \, t} \, \frac{(8\pi)^{\frac k2}\Gamma(\frac k2)}{\Gamma(k)} \,
    \langle f, f\rangle\, c_{\cE_2}(\phi_2)
    \prod_{p\neq 2}  \left\{ 1-p^{-3}-\frac{p-1}{p+1}p^{-k-1}a_p^2  \right\},
    \end{align}
where  $\langle f, f\rangle$ is the usual Petersson inner product and
    \[
    c_{\cE_2}(\phi_2)\coloneqq
    \frac{L(\frac12, \sigma_2\otimes\eta_{\cE_2})}{L(1, \sigma_2, \Ad)}
        =\begin{cases}
        \frac14 (3+2^{1-\frac k2}a_2)&
        \text{if $\cE_2\simeq\Q_2\times\Q_2$}, \\[5pt]
        \frac14(3-2^{1-\frac k2}a_2) & 
        \text{if $\cE_2\simeq\Q_2(\sqrt{5})$}, \\[5pt]
        \frac18(9-2^{2-k}a_2^2) &
        \text{otherwise}. 
        \end{cases}
    \]
The product in the right hand side is over odd primes.
\end{thm}

\begin{proof}
Set $\Gamma_\R(s)=\pi^{-\frac{s}{2}}\Gamma(\frac{s}{2})$ and $\Gamma_\C(s)=2(2\pi)^{-s}\Gamma(s)$,  where $\Gamma(s)$ is the usual gamma function.
Then we have $L(1, \pi, \Ad)=2^k\langle f, f\rangle$,  $\zeta_\Q(2)=\frac{\pi^2}{6}\cdot\Gamma_\R(2)=\frac{\pi}{6}$ and $c_\Q=|\Delta_\Q|=1$.
Also we see that
    \[
    \kappa_{\cE_2}(\sigma_2)
    =\frac{\frac43\cdot L(\frac12,  \sigma_2\otimes\eta_{\cE_2})}
    {2\cdot 2\cdot L(1,  \sigma_2,  \Ad)}=\frac13\cdot c_{\cE_2}(\phi_2).
    \]
Since $L(s,  \sigma_\infty\otimes\eta_{\cE_\infty})=\Gamma_\C(s+\frac{k-1}{2})$ and $L(s,  \sigma_\infty,  \Ad)=\Gamma_\C(s+k-1)\Gamma_\R(s+1)$,  one obtains
    \[
    \kappa_{\cE_\infty}(\sigma_\infty)
    =\frac{\Gamma_\R(2)\Gamma_\C(\frac{k}{2})}
    {2\Gamma_\C(k)\Gamma_\R(2)}
    =\frac{(2\pi)^{\frac{k}{2}}\Gamma(\frac{k}{2})}{2\Gamma(k)}.
    \]
The last equality for the constant $c_{\cE_2}(\phi_2)$ follows from
\begin{equation}\label{eq:rev1}
L(1,\pi_2,\mathrm{Ad})=\frac{8}{9-\lambda_2^2}, \quad L(\tfrac12,\pi_2\otimes\eta_{\cE_2})=\begin{cases} 2/(3-\lambda_2) & \text{if $\cE_2\simeq\Q_2\times\Q_2$}, \\ 2/(3+\lambda_2)  & \text{if $\cE_2\simeq\Q_2(\sqrt{5})$}, \\ 1 & \text{otherwise.} \end{cases}
\end{equation}
and $\lambda_2=2^{1-\frac{k}{2}}a_2$.
Substituting these into \eqref{eq:mvfL},  we obtain \eqref{eq:mvfhol1}.
\end{proof}

\begin{cor}\label{cor:mvfhol}
The limit
    \[
    \lim_{x\to\inf} x^{-1}
    \sum_E  L(\tfrac{1}{2},\pi\otimes\eta_E) 
    \]
exists, where the sum is over real quadratic fields $E$ with $|\Delta_E|<x$.
The value of the above limit equals  
    \[
    \frac{6(8\pi)^\frac k2\Gamma(\frac k2)}{\pi\, \Gamma(k)} \langle f, f \rangle
    \prod_{p}\left\{1-p^{-3}-\frac{p-1}{p+1}p^{-k-1}a_p^2 \right\},
    \]
where the product is over all primes.
\end{cor}

\begin{proof}
We see in Remark \ref{rem:dependence} that Corollary \ref{cor:mvfL} holds for any finite set $S$.
As in the proof of Theorem \ref{thm:mvfhol},  we specialize the mean value formula \eqref{eq:mvfL} to the current situation with $S=\{\infty\}$ and $\cE_\infty=\R\times\R$ to obtain the corollary.
\end{proof}


\subsection{Mean value theorem for algebraic modular forms}\label{sec:mvfalg}

In this subsection, we specialize Theorem \ref{thm:main} to the case of algebraic modular forms on the multiplicative group of a quaternion algebra over $F=\Q$.
Let $D$ be a quaternion division algebra over $\Q$ with odd prime discriminant $q$.
Take a maximal order $\cO$ in $D$ and set $\cO_v=\cO\otimes\Z_v$ for each finite place $v$.
The set $X(D)$ consists of imaginary quadratic fields which do not split at $v=q$.
By \cite{SWY}*{Corollary 4.22},  we may take $X(D)$ so that $\fo_E\subset\cO$ for all but finitely many $E\in X(D)$,  where $\fo_E$ is the ring of integers of $E$.
Note that we can ignore the contribution of finite number of quadratic fields to the mean value formula.
With a slight abuse of notation,  we identify $X(D)$ with the set of $E\in X(D)$ such that $\fo_E\subset\cO$.

We summarize some notation about algebraic modular forms.
We refer the reader to \cite{Voight} and \cite{SWY}*{\S\,2} for more details.
For a finite place $v$, let $U_v$ be the image of $\cO_v^\times$ under the natural projection $D_v^\times\rightarrow PD_v^\times$ and $U=\prod_{v<\infty} U_v$.
These are maximal compact subgroups of $PD_v^\times$ and $PD_{\A_\fin}^\times$, respectively.
Let $\mathcal{A}(\cO)$ be the space of complex functions on the set of ideal classes $\mathrm{Cl}(\cO)=PD^\times\bs PD^\times_\A/PD^\times_\inf U$ for $\cO$.
Take $1=x_1$, $x_2,\dots, x_h \in PD^\times_{\A_\fin}$ so that we have a disjoint decomposition
    \[
    PD_\A^\times=\ds\amalg_{i=1}^h PD^\times x_i PD_\infty^\times U.
    \]
Elements of $\mathcal{A}(\cO)$ are functions on the finite set $\{x_1, \ldots, x_h\}$.
We define a $PD_\A^\times$-invariant inner product on $\mathcal{A}(\cO)$ by
    \[
    (\phi,\phi')=\sum_{j=1}^h \frac{\phi(x_j)\, \overline{\phi'(x_j)}}{w_j}, \hspace{15pt}
    \phi,\phi'\in \mathcal{A}(\cO),
    \]
where $w_j$ is the order of the finite group $PD^\times\cap x_jUx_j^{-1}$.

Hecke operators on $\mathcal{A}(\cO)$ are defined similarly as the holomorphic modular forms.
A Hecke eigenform $\phi\in\mathcal{A}(\cO)$ gives rise to an automorphic form on $PD_\A^\times$, which we again denote by $\phi$.
By Eichler mass formula \cite{Voight}*{Theorem 25.3.15}, the Petersson inner product \eqref{eq:innprod} becomes
    \[
    \langle \phi, \phi' \rangle=\frac{12}{q-1}(\phi, \phi').
    \]
Let $\pi=\otimes_v \pi_v$ be the corresponding automorphic cuspidal representation of $PD_\A^\times$.
Note that $\pi_\inf$ is the trivial representation,  and $\pi_q$ is the trivial or the unramified quadractic character since $\pi_q|_{U_q}$ is trivial and $PD^\times_q/U_q\cong \Z/2\Z$. In addition,  $\pi_p$ is unramified for any places $p\neq q, \infty$ since they are trivial on $U_p$.
Hereafter, we assume that $L(\tfrac{1}{2},\pi)\neq 0$. 
Then we see from the local root numbers that $\pi_q$ is the trivial representation.
Set $S=\{2, q, \infty\}$.
The cusp form $\phi$ is factorizable, i.e. $\phi=\otimes_v\phi_v$ with $\phi_v\in\pi_v$.
We assume that $\phi_v$ is the normalized spherical vector for all $v\not\in S$ and $v=2$.

Let $\mathrm{Cl}(E)$ denote the ideal class group of $E$.
Since we assumed that $\fo_E\subset\cO$,  we have a map 
    \begin{equation}\label{eq:classmap}
    \mathrm{Cl}(E)=
    E^\times\bs\A_E^\times/
    ( E_\infty^\times\cdot\prod_{v<\infty}\fo_{E, v}^\times )
    \rightarrow 
    PD^\times\bs PD_\A^\times/PD_\infty^\times U=\mathrm{Cl}(\cO).
    \end{equation}
By abuse of notation,  we write the function on $\mathrm{Cl}(E)$ which is the composition of $\phi$ with the above map again by $\phi$. 

Recall that the positive integer $a_E$ is attached to each quadratic field $E$ in \eqref{eq:a_E}.

\begin{thm}\label{thm:mvfalg}
Set $S=\{2, q, \infty\}$.
We take a triple $\cE_S=(\cE_v)_{v\in S}$ so that $\cE_\infty=\C$ and $\cE_q$ is the unramified quadratic extension of $\Q_q$.
Then
    \begin{align}\label{eq:mvfalg}
    &\hspace{-30pt}\lim_{x\to\inf} \, x^{-\frac32}        
    \sum_{\substack{E\in X(D, \cE_S) , \\ a_E<x}} 
    \left|\sum_{t\in\mathrm{Cl}(E)}\phi(t) \right|^2  \\
    &=\frac{1}{6\pi q}\, \frac{q-1}{q+1} \,  L_\fin(\tfrac{1}{2},\pi)  (\phi, \phi)  \,
    \Omega(\cE_2) \,
    \prod_{p\neq 2,q}  \left\{ 1-p^{-3}-\frac{p-1}{p+1}p^{-3}\lambda_p^2  \right\}.
     \nonumber
    \end{align}
Here,  $\phi$ in the left hand side is seen as a function on $\mathrm{Cl}(E)$ by the map \eqref{eq:classmap}. 
The product in the right hand side is over odd primes other than $q$.
The factor $\Omega(\cE_2)$ is given as follows:
 \begin{align*}
    \Omega(\cE_2)=&4N(\ff_{\cE_2})^{\frac12}\cdot
    \frac{L(\frac12,  \pi_2\otimes\eta_{\cE_2})}{L(1,  \pi_2,  \Ad)}  \\
        =&\begin{cases}
        3+\lambda_2 & \text{if $\cE_2\simeq\Q_2\times \Q_2$}, \\ 
        3-\lambda_2 & \text{if $\cE_2\simeq\Q_2(\sqrt{5})$}, \\  
        9-\lambda_2^2 & \text{if $\cE_2\simeq\Q_2(\sqrt{3})$ or  $\cE_2\simeq\Q_2(\sqrt{7})$}, \\   \sqrt{2}(9-\lambda_2^2) & \text{if $\cE_2\simeq\Q_2(\sqrt{\pm 2})$   or $\cE_2\simeq\Q_2(\sqrt{\pm 10})$}.
        \end{cases}
  \end{align*}
\end{thm}
\begin{proof}
First we consider the left hand side of \eqref{eq:mvf}.
For $E\in X(D, \cE_S)$,  note that $-a_E N(\ff_{\cE_2})$ equals the fundamental discriminant of $E$. 
Since $\cE_q$ is unramifed over $\Q_q$, we have $|d_E|_q=1$ and $a_E=N(\ff_E^S)$. 
We have $L(1, \eta_E)=(a_E N(\ff_{\cE_2}))^{-\frac12}|\mathrm{Cl}(E)|$ from the Dirichlet class number formula, since we may ignore $E=\Q(\sqrt{-1})$ and $\Q(\sqrt{-3})$.
Hence we get
    \[
    \frac{1}{2} \,(a_E N(\ff_{\cE_2}))^{\frac12}L(1,\eta_E)\, \cP_E(\phi)
    = \frac{|\mathrm{Cl}(E)|}{2} 
    \int_{E^\times \A_\Q^\times\bs\A_E^\times} \phi(h) \, \d^\times h
    =\sum_{t\in\mathrm{Cl}(E)}\phi(t). 
    \]
Recall that $\d^\times h$ is the Tamagawa measure and $\vol(E^\times \A_F^\times\bs  \A_E^\times, \d^\times h)=2$.
Since we prove $\alpha_{E,S}(\phi,\phi)\neq 0$ for any $E$ below, the left hand side of \eqref{eq:mvf} $(t=3/2)$ becomes
    \[
    4\, N(\ff_{\cE_2})^{-1}\cdot \lim_{x\to\infty}\, x^{-\frac32}
    \sum_{\substack{E\in X(D, \cE_S) \\ a_E<x}}
    a_E^{-\frac12} 
    \alpha_{E, S}(\phi)^{-1} \left|\sum_{t\in\mathrm{Cl}(E)}\phi(t) \right|^2.
    \]
    
Recall that $\pi_q$ and $\pi_\infty$ are the trivial representations and we assumed $\fo_E\subset\cO$.  
From Lemma \ref{lem:fP},  Lemma \ref{lem:omg} and Corollary \ref{cor:fPun}, we get
\[
\alpha_{E_2}(\phi_2,  \phi_2)=\frac{\zeta_{\Q_2}(2)L(\frac12, \pi_2)L(\frac12, \pi_2\otimes\eta_{\cE_2})}
    {L(1, \pi_2, \Ad)L(1, \eta_{\cE_2})} \times \begin{cases} 2 & \text{if $E_2\simeq \Q_2\times\Q_2$ or $\Q_2(\sqrt{5})$}, \\ 1 & \text{otherwise}. \end{cases}
\]
Here, we note that $d_{E_2} \in \Z_2^\times \sqcup 2\Z_2^\times$ and $\eta_{E_2}=\eta_{\cE_2}$. 
Furthermore,
\[
\alpha_{E_q}(\phi_q,  \phi_q)=\int_{\Q_q^\times\bs E_q^\times} \frac{\d h_{E_q}}{\d^\times z_q}\times \langle\phi_q,\phi_q\rangle=\langle\phi_q,\phi_q\rangle,
\]
\[
\frac{\alpha_{E_\inf}(\phi_\inf,  \phi_\inf)}{ \langle\phi_\inf,\phi_\inf\rangle }=\int_{\R^\times\bs \C^\times} \frac{\d h_{\C}}{\d^\times z_\inf} =\frac{1}{\pi}\int_{\R} \frac{\d a}{a^2+|d_E|} =|d_E|^{-\frac12} .
\]
Hence, $\alpha_{E,S}(\phi,\phi)\neq 0$ holds for any $E$. 
We also note that
\begin{equation}\label{eq:rev2}
\zeta_{\Q_2}(2)=\frac{4}{3} ,\quad N(\ff_{\cE_2})=\begin{cases}
        1 & \text{if $\cE_2\simeq\Q_2\times \Q_2$ or $\cE_2\simeq\Q_2(\sqrt{5})$}, \\  
        4 & \text{if $\cE_2\simeq\Q_2(\sqrt{3})$ or  $\cE_2\simeq\Q_2(\sqrt{7})$}, \\  
        8 & \text{otherwise}
        \end{cases}
\end{equation}
and \eqref{eq:rev1} remains valid in this case.
Therefore,  by the above equalities and $\alpha_{E,  S}(\phi)=\alpha_{E_\inf}(\phi_\inf,  \phi_\inf)\alpha_{E_2}(\phi_2,  \phi_2)\alpha_{E_q}(\phi_q,  \phi_q)$,  we get
    \begin{align*}
     4N(\ff_{\cE_2})^{-1}a_E^{-\frac12}\alpha_{E,  S}(\phi)^{-1}
     &=\frac{L(1, \eta_{\cE_2})}{L(\frac12, \pi_2)}\cdot
     \frac{2 \, L(1, \pi_2, \Ad) N(\ff_{\cE_2})^{-\frac12}}
     {\zeta_{\Q_2}(2)L(\frac12, \pi_2\otimes\eta_{\cE_2})\cdot
     \langle \phi, \phi \rangle} \\
     &=\frac{L(1, \eta_{\cE_2})}{L(\frac12, \pi_2)}\cdot 
         \frac{6}{\Omega(\cE_2)\cdot \langle \phi, \phi \rangle}.
    \end{align*}
Thus, the left hand side of \eqref{eq:mvf} $(t=3/2)$ equals
    \[
    \frac{q-1}{2\, (\phi, \phi)}\, \frac{L(1, \eta_{\cE_2})}
    {L(\frac12, \pi_2)\Omega(\cE_2)}\cdot
     \lim_{x\to\infty}\, x^{-\frac32}\sum_{\substack{E\in X(D, \cE_S) \\ a_E<x}}
      \left|\sum_{t\in\mathrm{Cl}(E)}\phi(t) \right|^2.
    \]
For the constant $\Omega(\cE_2)$, its value follows from \eqref{eq:rev1} and \eqref{eq:rev2}. 

On the other hand,  the right hand side of \eqref{eq:mvf} equals
    \[
    \frac{1}{12\pi}\, \frac{(q-1)^2}{q(q+1)} \, L_\fin(\tfrac12, \pi) 
    \frac{L(1, \eta_{\cE_2})}{L(\frac12, \pi_2)}
    \prod_{p\neq 2,q}  \left\{ 1-p^{-3}-\frac{p-1}{p+1}p^{-3}\lambda_p^2  \right\}. 
    \]
Putting altogether, we obtain \eqref{eq:mvfalg}.
\end{proof}


\section{Prehomogeneous zeta function}\label{sec:pvzeta}


\subsection{Definition and basic properties}\label{sec:definition}

We recall the definition and basic properties of the prehomogeneous zeta function with toric periods introduced in the previous paper \cite{SW}.


\subsubsection{Prehomogeneous vector space}
Let $F$ be a field of characteristic zero and $D$ a quaternion algebra over $F$.
Set $G=D^\times\times D^\times \times \GL_2$ and $V=D\oplus D$.
We consider the $F$-rational right action $\rho\,\colon G\rightarrow\GL(V)$ of $G$ on $V$ given by
\[
(x,y) \cdot \rho(g_1,g_2,g_3) = (g_1^{-1}xg_2,g_1^{-1}yg_2)g_3, \quad (g_1,g_2,g_3)\in G, \quad (x,y)\in V.
\]
The triple $(\rho, G, V)$ forms a prehomogeneous vector space.
The fundamental relative invariant is $P(x,y)=-\det(x\, y^\iota-y \, x^\iota)$ and the corresponding character of $G$ is $\omega(g)=\det(g_1)^{-2}\det(g_2)^2\det(g_3)^2$, i.e.
    \[
    P((x,y)\cdot\rho(g))=\omega(g)\, P(x,y), \hspace{25pt}  g\in G, \ (x,y)\in V.
    \]
It is easy to check that
    \[
    \mathrm{Ker}\rho=\{ (a,b,ab^{-1}I_2)\in G \mid a,b\in\GL_1  \}
    \cong \GL_1\times \GL_1,
    \]
where $I_n$ is the unit matrix in $M_n$.
Set $H=\mathrm{Ker}\rho\bs G$.
By abuse of notation, we will also let $\rho$ denote the induced action of $H$ on $V$.

We define a non-degenerate bilinear form $\langle \; , \; \rangle$ on $V$ by
    \[
    \langle (x_1,y_1) , (x_2,y_2) \rangle\coloneqq\Tr(x_1 x_2)+\Tr(y_1  y_2), \hspace{25pt} 
    (x_i, y_i)\in V, \ i=1, 2.
    \]
Let $(\rho^\vee, V)$ denote the contragredient representation of $\rho$ on $V$, namely $\langle x\cdot \rho(g) , y\cdot \rho^\vee(g) \rangle=\langle x , y \rangle$. 
One can see that $\rho^\vee$ is given by
    \[
     (x,y)\cdot \rho^\vee(g_1,g_2,g_3) 
     =(g_2^{-1} x g_1,g_2^{-1}yg_1)\, ^t{}\!g_3^{-1}, \quad
     (g_1, g_2, g_3)\in G, \quad (x, y)\in V
     \]
and $\mathrm{Ker}\rho=\mathrm{Ker}\rho^\vee$.

For each point $x\in V(F)$, we say that $x$ is regular if the $G(\bar{F})$-orbit of $x$ is Zariski open dense in $V(\bar{F})$, and $x$ is singular otherwise.
The set of regular points is $V^0(F)=\{ x\in V(F) \mid P(x)\neq0 \}$.
For $\cE\in X(D)$, $V_{\cE}(F)= \{ x\in V \mid  P(x)\in d_{\cE}(F^\times)^2 \}$ is a $G(F)$-orbit and $\cE\mapsto V_{\cE}(F)$ gives rise to a bijective correspondence between $X(D)$ and the set of $G(F)$-orbits in $V^0(F)$.
We take a base point $x_{\cE}=(1, \delta_\cE)\in V_{\cE}(F)$.


\subsubsection{Zeta functions}\label{sec:20200823}

For the rest of this section, $F$ is a number field.
We take an $F$-rational gauge form $\omega$ on $H$. 
Then the Tamagawa measure $\d h$ on $H(\A)$ is written as
    \[
    \d h= c_F^{-1}\, |\Delta_F|^{-5}\prod_{v\in\Sigma} c_v \, \omega_v
    \]
where $\omega_v$ is the measure on $H(F_v)$ obtained from $\omega$.
The constants $c_v$ and $c_F$ are defined in \S\,\ref{sec:n2}.

Let $\cS(V(\A))$ denote the Schwartz space on $V(\A)$. 
Let $\pi$ be a cuspidal automorphic representation of $PD^\times_\A$ which is not $1$-dimensional.
For $s\in\C$, $\phi=\otimes_v\phi_v\in\pi$ and $\Phi\in\cS(V(\A))$, we define the global zeta functions $Z(\Phi,\phi,s)$ and $Z^\vee(\Phi,\phi,s)$:
    \begin{align*}
    Z(\Phi,\phi,s)&=\int_{H(F)\bs H(\A)} |\omega(h)|^s \phi(g_1)\, \overline{\phi(g_2)}
    \sum_{x\in V^0(F)} \Phi(x\cdot \rho(h))\, \d h,\\
    Z^\vee(\Phi,\phi,s)&=\int_{H(F)\bs H(\A)} |\omega(h)|^{-s} \phi(g_1)\, 
    \overline{\phi(g_2)} \sum_{x\in V^0(F)} \Phi(x\cdot \rho^\vee(h))\, \d h.
    \end{align*}

We set
    \begin{align*}
    Z_+(\Phi,\phi,s)&=\int_{H(F)\bs H(\A) , |\omega(h)|\geq 1} |\omega(h)|^s \phi(g_1)\, 
    \overline{\phi(g_2)} \sum_{x\in V^0(F)} \Phi(x\cdot \rho(h))\, \d h , \\
    Z_+^\vee(\Phi,\phi,s)&=\int_{H(F)\bs H(\A),|\omega^\vee(h)|\geq 1} 
    |\omega(h)|^{-s} \phi(g_1)\, \overline{\phi(g_2)} \sum_{x\in V^0(F)} 
    \Phi(x\cdot \rho^\vee(h))\, \d h .
    \end{align*}
These integrals define entire functions if they absolutely converge for sufficiently large $\Re(s)$ (cf. \cite{Kimura}*{Proposition 5.15}).

Let
    \[
    f_{\phi}(g')=\int_{Z(\A)D^\times \bs D_\A^\times}\phi(g)\, \overline{\phi(gg')} \, 
    \frac{\d g}{\d z}, \hspace{25pt} g'\in D_\A^\times
    \]
be the matrix coefficient.
For $\Psi\in\cS(D_\A)$, the Godement-Jacquet zeta integral is the meromorphic continuation of
    \[
    Z^\GJ(\Psi,f_{\phi},s)= \int_{D_\A^\times} \Psi(g) \, f_{\phi}(g) \, |\det(g)|^s \, \d g.
    \]
Given a Schwartz function $\Phi\in\cS(V(\A))$ , we define $\Phi_K\in\cS(V(\A))$ and $\Phi_1\in\cS(D_\A)$ by
    \[
    \Phi_K(x)=\int_K\Phi(x\cdot\rho(1, 1, k)) \d k, \qquad 
    \Phi_1(x)\coloneqq\Phi(x, 0).
    \]
Note that we have $\sF(\Phi_K)=(\sF\Phi)_K$, where 
    \[
    \sF\Phi(z')=\int_{V(\A)}\Phi(z)\psi_F(\langle z, z' \rangle) \d z
    \]
is the Fourier transform of $\Phi$.
    
The basic analytic properties of $Z(\Phi,\phi,s)$ are summarized as follows.

\begin{thm}\label{thm:global}
\begin{itemize}
\item[(1)] The zeta function $Z(\Phi,\phi,s)$ is absolutely convergent and holomorphic on the domain $\{s\in\C \mid \Re(s)>T\}$ for sufficiently large $T>0$.

\item[(2)] For $\Re(s)>T$ and $\Phi \in \cS(V(\A))$, we have
    \begin{align*}
    Z(\Phi,\phi,s)=& Z_+(\Phi, \phi, s)+Z_+^\vee(\sF\Phi, \phi, 2-s)  \\
    &+c_F \, \fc_F \frac{Z^\GJ(((\sF\Phi)_K)_1, \overline{f_{\phi}}, 1)}{2s-3}
    -c_F\, \fc_F \frac{Z^\GJ(\left(\Phi_K\right)_1, f_{\phi}, 1)}{2s-1},
    \end{align*}
    where we set
    \[
    \fc_F \coloneqq \frac{2\pi^{r_2}c_F}{|\Delta_F|^{\frac12} \, \zeta_F(2)}.
    \]
    
\item[(3)] The zeta function $Z(\Phi,\phi,s)$ is meromorphically continued to the whole $s$-plane.
The possible poles are at most simple, located at $s=\frac12$ and $s=\frac32$.
In addition, the following functional equation holds:
    \begin{equation}\label{eq:gf2}
    Z(\Phi,\phi,s)=Z^\vee(\sF\Phi, \phi, 2-s).
    \end{equation}
\end{itemize}
\end{thm}
\begin{proof}
    The assertions were proved in \cite{SW}*{Lemma 3.1, Theorem 3.2 and Corollary 3.3 in \S 3}. 
    Note that the measure $\fc_F\d h$ on $H(\A)$ is used in \cite{SW}*{\S 3} and its normalization is different from the other sections in \cite{SW}. 
    We also note that the assertion (3) above follows from (1) and (2). 
\end{proof}

We also introduce the local zeta function
    \[
    Z_{\cE_v}(\Phi_v, \phi_v, s)
    =\frac{2\,c_v}{L(1, \eta_{\cE_v})^2} \int_{V_{\cE_v}(F_v)}
    \alpha_{\cE_v}(\pi_v(g_1)\phi_v, \pi_v(g_2)\phi_v)\, |P(x)|_v^{s-2} \, \Phi_v(x) \d x,
    \]
where $\cE_v\in X(D_v)$ and $x=x_{\cE_v}\cdot \rho(g_1,g_2,g_3)$. 
Since we have 
\[
\alpha_{\cE_v}(\pi_v(t_1)\varphi_v, \pi_v(t_2)\varphi'_v)=\alpha_{\cE_v}(\varphi_v,\varphi'_v) 
\]
for any element $(t_1,t_2,t_3)$ of the stabilizer of $x_{\cE_v}$ in $G(F_v)$, the above integral is well-defined.  

\begin{lem}\label{lem:loc}
Let the notation be as above. 
\begin{itemize}
\item[(1)] The local zeta function $Z_{\cE_v}(\Phi_v, \phi_v, s)$ is absolutely convergent when $\Re(s)$ is sufficiently large and meromorphically continued to the whole $s$-plane.

\item[(2)] Assume that $\alpha_{\cE_v}(\phi_v,\phi_v)\neq0$.
For $s_0\in\C$ there exists a test function $\Phi_v\in \cS(V(F_v))$ supported on $V_{\cE_v}(F_v)$ such that $Z_{\cE_v}(\Phi_v, \phi_v, s)$ is entire and non-zero around $s=s_0$.
\end{itemize}
\end{lem}

\begin{proof}
(1) is proved in \cite{Li4}.
A test function supported on a sufficiently small neighborhood of $x_{\cE_v}$ satisfies the conditions of (2).
\end{proof}

Take two quadratic \'etale subalgebras $\cE_v$ and $\cE_v'$ of $F_v$ in $D_v$. 
Suppose $\cE_v\simeq \cE_v'$. 
Then, there exist $t\in D_v^\times$ and $a\in F_v^\times$ such that $at^{-1}\delta_{\cE_v}t=\delta_{\cE_v'}$. 
This implies $x_{\cE_v'}=x_{\cE_v}\cdot \rho(t,t,\diag(1,a))$ and $a^2d_{\cE_v}=d_{\cE_v'}$. 
For $\phi_v,  \phi'_v\in\pi_v$, we have
\begin{equation}\label{eq:alc}
    \alpha_{\cE_v'}(\phi_v,  \phi'_v)=|a|_v^{-1}\int_{F_v^\times\bs t^{-1}\cE_v^\times t}\langle \pi_v(h_v)\phi_v,  \phi'_v \rangle_v \d h_v
    =\left|\frac{d_{\cE_v}}{d_{\cE_v'}}\right|_v^\frac12\alpha_{\cE_v}(\pi(t)\phi_v,  \pi(t)\phi'_v).
\end{equation}
In particular, we get $Z_{\cE_v'}(\Phi_v, \phi_v, s)=\left|\frac{d_{\cE_v}}{d_{\cE_v'}}\right|_v^{\frac12} Z_{\cE_v}(\Phi_v, \phi_v, s)$ for $v\in S$, since for any $t$ in $D^\times$ we have $Z_{\cE_v}(\Phi_v,\pi_v(t) \phi_v, s)=Z_{\cE_v}(\Phi_v, \phi_v, s)$ by change of variable as $x \mapsto x\cdot \rho(t^{-1},t^{-1},1)$.

From now till the end of the next section, we assume the following. 
\begin{quote}\textbf{Assumption.}
We fix a decomposable element $\phi=\otimes_v\phi_v\in\pi$,  finite set of places $S\subset\Sigma$ and $\cE_S=(\cE_v)_{v\in S}\in X(D_S)$ so that Condition \ref{condition} is satisfied and $\alpha_{\cE_v}(\phi_v,\phi_v)\neq0$ for all $v\in S$ as in Theorem \ref{thm:main}. 
\end{quote}
One can see that the left hand side of \eqref{eq:mvf} depends only on the isomorphism classes of $\cE_v$'s.
When proving Theorem \ref{thm:main}, we may fix the choice of an \'etale subalgebra $\cE'_v\subset D_v$ as follows.
\begin{itemize}
\item For any $v\notin S$ and any $\cE_v'\in X(D_v)$, we also suppose that $d_{\cE_v'}\in \fo_v \setminus \varpi_v^2\fo_v$ and the maximal compact subgroup of $\cE_v^{\prime\times}$ is contained in $K_v$. 
\item For each archimedean place $v$, we have $d_{\cE_v}=\delta_{\cE_v}^2=\pm 1$. 
\end{itemize} 

We can impose the second condition, since we have $\alpha_{\cE_v}(\pi_v(t_v)\phi_v,\pi_v(t_v)\phi_v)\neq 0$ for some $t_v\in D_v^\times$ by \eqref{eq:alc} and we may replace $\pi_v(t_v)\phi_v$ by $\phi_v$ without loss of generality (cf. Remark \ref{rem:mvf}). 
The following lemma is obtained from the first condition.  
\begin{lem}\label{lem:unchange}
Take a finite subset $T$ of $\Sigma$ containing $S$. 
For $\cE'=(\cE'_v)_{v\in T\setminus S}\in X(D_{T\setminus S})$, let $\cE_S\cup\cE'$ be the join of $\cE_S$ and $\cE'$, which is an element of $X(D_T)=\prod_{v\in T}X(D_v)$. 
Then, we have 
\[
\alpha^{\cE_S\cup\cE'}_E(\phi)|\cP_E(\phi)|^2=\alpha^{\cE_S}_E(\phi)|\cP_E(\phi)|^2
\]
for any $\cE'\in X(D_{T\setminus S})$.    
\end{lem}
\begin{proof}
For $v\in T\setminus S$ and $\cE_v'\in X(D_v)$, we get $\alpha^\#_{\cE_v'}(\phi_v,\phi_v)=1$ by the above condition (cf. Corollaries \ref{cor:fP} and \ref{cor:fPun}). 
Then, the assertion can be proved by the definition \eqref{eq:alphaEv} of $\alpha^{\cE_S}_E(\phi)|\cP_E(\phi)|^2$.
\end{proof}


We set 
\[
\xi(D, \cE_S, \phi, s)=c(D, \phi, S, s) \sD(D, \cE_S, \phi, s)
\]
with a meromorphic function
    \begin{equation}\label{eq:rel814}
    c(D, \phi, S, s)=\frac{ 1 }{2 \, |\Delta_F|^4  \, c_F }  \frac{\zeta^S_F(2s-1) \, 
    L^S(2s-1,\pi,\mathrm{Ad}) }{\zeta^S_F(2)^3\, \alpha_{\cE_S}(\phi, \phi)}
    \end{equation}
and a Dirichlet series
    \begin{equation}\label{eq:rel815}
    \sD(D, \cE_S, \phi, s)=\sum_{E\in X(D,\cE_S)} 
    \frac{L(1,\eta)^2 \, \alpha_E^{\cE_S}(\phi)\,  |\cP_E(\phi)|^2\, \mathcal{D}_E^S(\pi,s)}
    {N(\ff_E^S)^{s-1}}.
    \end{equation}
Here, $\mathcal{D}_E^S(\pi,s)=\prod_{v\notin S}\mathcal{D}_{E_v}(\pi_v,s)$ is the product of
    \begin{equation*}
    \mathcal{D}_{E_v}(\pi_v,s)= 
        \begin{cases} 
        1+q_v^{-2s+1}+q_v^{-2s}+q_v^{-4s+1} -2\eta_v(\varpi_v)\, q_v^{-2s}\lambda_v 
        & \text{if $\eta_v$ is unramified}, \\ 
        1+q_v^{-2s+1}
        & \text{if $\eta_v$ is ramified}.
        \end{cases}
    \end{equation*}
Note that $c(D,  \phi,  S,  s)$ is holomorphic and does not vanish in the region $\Re(s)>1$ (see \cite{JS}*{Theorem 5.3}). 
Choose a maximal order $\cO$ in $D$, and an integral structure of $V$ is given by $V(\fo)=\cO\oplus\cO$. 
The next theorem is a special case of the explicit formula \cite{SW}*{Theorem 1.2} for $Z(\Phi, \phi, s)$, which is sufficient for our purpose.
\begin{thm}\label{thm:Dirichlet}
Take a decomposable Schwartz function $\Phi=\otimes_v \Phi_v\in\cS(V(\A))$ as follows.
For $v\notin S$, $\Phi_v$ is the characteristic function of $V(\fo_v)$.
For $v\in S$, we choose $\Phi_v$ so that 
    \begin{equation}\label{eq:cond}
     Z_{\cL_v}(\Phi_v, \phi_v, s)\equiv 0 \quad 
     \text{for all $\cL_v\in X(D_v)\setminus\{\cE_v\}$}.
    \end{equation}
For $\Re(s)>0$ sufficiently large, the Dirichlet series $\xi(D,\cE_S, \phi, s)$ converges absolutely and satisfies
    \begin{equation}\label{eq:gloloc}
    Z(\Phi, \phi, s)= \left(\prod_{v\in S} Z_{\cE_v}(\Phi_v, \phi_v, s)\right)
     \,\xi(D, \cE_S, \phi, s).
    \end{equation}
\end{thm}


\subsection{Residues of the Dirichlet series $\sD(D, \cE_S, \phi, s)$}\label{sec:zeta}

Throughout this subsection,  let $\Phi=\otimes_v \Phi_v\in\cS(V(\A))$ and assume that $\Phi_v$ is the characteristic function of $V(\fo_v)$ for all $v\notin S$ as in Theorem \ref{thm:Dirichlet}. 
The goal of this subsection is to compute the residue of the Dirichlet series $\sD(D, \cE_S, \phi, s)$ at $s=\tfrac32$, which is the key ingredient in the proof of Theorem \ref{thm:main}.
    
\begin{lem}\label{lem:analytic}
We keep the assumptions and notation of the previous subsection.

\begin{itemize}
\item[(1)] The function $\xi(D, \cE_S, \phi, s)$ is meromorphically continued to $\C$ and the possible poles are at most simple, located at $s=\tfrac12$ and $\tfrac32$.

\item[(2)] The Dirichlet series $\sD(D, \cE_S, \phi, s)$ is meromorphically continued to $\C$ and $\sD(D, \cE_S, \phi, s)$ is holomorphic at $s=1$.
\end{itemize}
\end{lem}

\begin{proof}
(1) The meromorphic continuation follows from that of $Z(\Phi, \phi, s)$, see Theorem \ref{thm:global}.
By Lemma \ref{lem:loc} (2), we may assume $Z_{\cE_v}(\Phi_v, \phi_v, \tfrac12)\neq0$  and hence the possible pole at $s=\frac12$ is at most simple.
So as the possible pole at $s=\frac32$.

(2) The first half is obvious from (1). 
Since it is known that $\zeta^S_F(2s-1)\, L^S(2s-1,\pi,\Ad)$ has a simple pole at $s=1$ (cf. \cite{Bump}*{p.374}), $c(s,D,\phi,S)$ has a simple pole at $s=1$. 
By (1) we find that $\xi(D, \cE_S, \phi, s)$ is holomorphic at $s=1$.
Therefore, $\sD(D, \cE_S, \phi, s)$ is holomorphic at $s=1$.
\end{proof}

The following result of Blomer and Brumley \cite{BB} will be used repeatedly.

\begin{thm}\label{thm:BB}
For $v\notin S$, we have $q_v^{-\frac{7}{64}}\leq|\alpha_v|\leq q_v^{\frac{7}{64}}$ and in particular $|\lambda_v|\leq q_v^{\frac12}(q_v^{\frac{7}{64}}+q_v^{-\frac{7}{64}})$.
\end{thm}

\begin{lem}\label{lem:non-negative}
The factor $\alpha_{E}^{\cE_S}(\phi) |\cP_E(\phi)|^2$ is non-negative for any $E\in X(D,\cE_S)$.     
\end{lem}
\begin{proof}
    Take an element $E\in X(D,\cE_S)$ such that $\pi$ is $E^\times$-distiguished. 
    Under Condition \ref{condition}, there exists an element $g_S=(g_v)_{v\in S}\in D_S^\times$ so that $\cP_E(\pi(g_S)\phi)\neq 0$. 
    Hence, $|\cP_E(\pi(g_S)\phi)|^2>0$ and by \eqref{eq:Wald} we have $\prod_{v\in S}\alpha_{E_v}(\pi_v(g_v)\phi_v, \pi_v(g_v)\phi_v)\neq 0$.  
    Therefore,
    \[
    \frac{\alpha_{E}^{\cE_S}(\phi) |\cP_E(\phi)|^2}{|\cP_E(\pi(g_S)\phi)|^2} =  \prod_{v\in S} \frac{\alpha_{\cE_v}(\phi_v,\phi_v) \, |d_{\cE_v}|_v }{\alpha_{E_v}(\pi_v(g_v)\phi_v, \pi_v(g_v)\phi_v) \, |d_{E_v}|_v}.
    \]
Fix a place $v\in S$. 
By \eqref{eq:alc}, we take $u_v\in D_v^\times$ and $a_v\in F_v^\times$ so that 
    \[
    \alpha_{E_v}(\pi_v(g_v)\phi_v, \pi_v(g_v)\phi_v)=|a_v|_v^{-1}\alpha_{\cE_v}(\pi_v(u_vg_v)\phi_v, \pi_v(u_vg_v)\phi_v)
    \]
Set
    \[
    \beta_{\cE_v}(\phi_v : g_1, g_2)\coloneqq
    \frac{\alpha_{\cE_v}(\pi_v(g_1)\phi_v, \pi_v(g_2)\phi_v)}{\alpha_{\cE_v}(\phi_v, \phi_v)}, 
    \quad \beta_{\cE_v}(\phi_v : g)\coloneqq\beta_{\cE_v}(\phi_v:g,1)
    \]
    for $g_1, g_2, g\in D_v^\times$.
In the proof of \cite{SW}*{Lemma 5.1}, we observed
    \[
    \beta_{\cE_v}(\phi_v: g_1, g_2)=\beta_{\cE_v}(\phi_v: g_1) \, 
    \overline{\beta_{\cE_v}(\phi_v: g_2)}.
    \]
Thus, 
    \[
    \frac{\alpha_{\cE_v}(\phi_v,\phi_v)}{\alpha_{E_v}(\pi_v(g_v)\phi_v, \pi_v(g_v)\phi_v)} = \frac{|a_v|_v}{|\beta_{\cE_v}(\phi_v: u_vg_v)|^2}>0.
    \]
This completes the proof.
\end{proof}
\begin{lem}\label{lem:propzeta}
Assume that $S$ contains all places above primes less than $14$.
If we write Dirichlet series $\sD(D, \cE_S, \phi, s)$ as
    \[
     \sD(D, \cE_S, \phi, s)=\sum_{n=1}^\inf \frac{a(D, \cE_S, \phi, n)}{n^s},
    \]
the constants $a(D, \cE_S, \phi, n)$ are non-negative real numbers and non-zero for infinitely many $n$ unless $\sD(s, D, \cE_S, \phi)$ is identically zero.
\end{lem}
\begin{proof}
By Lemma \ref{lem:non-negative}, it is sufficient to consider the factor $\mathcal{D}_E^S(\pi,s)$. 
For $v\not\in S$, set 
    \[  \left\{
        \begin{array}{lll}
        A_v=1+q_v^{-1} -2\eta_v(\varpi_v)\, q_v^{-1}\lambda_v,&  B_v=1 
        & \text{if $\eta_v$ is unramified} \\
        A_v=1, & B_v=0 
        & \text{if $\eta_v$ is ramified}.
        \end{array} \right.
    \]
From Theorem \ref{thm:BB}, we see that $A_v>0$.
It follows from
    \begin{equation}\label{eq:de1}
    \mathcal{D}_E^S(\pi,s)=\prod_{v\not\in S} (1+A_v q_v^{-2s+1}+B_v q_v^{-4s+1}) 
    \end{equation}
that  $a(D,\cE_S, \phi, n)\geq0$.
The second assertion follows from \cite{SW}*{Theorem 1.7}.
\end{proof}

\begin{lem}\label{lem:dir}
Suppose that $\sD(D, \cE_S, \phi, s)$ is not identically zero.
Then, $\sD(D, \cE_S, \phi, s)$ is absolutely convergent for $\Re(s)>\tfrac{3}{2}$, has a simple pole at $s=\tfrac{3}{2}$, and is holomorphic on the domain $\{s\in\C \mid \Re(s)>1$, $ s\neq \tfrac{3}{2} \}$.
In particular, $\fC(D,\cE_S,\phi)=\underset{s=\frac32}{\Res}\,\sD(D, \cE_S, \phi, s)>0$.
\end{lem}

\begin{proof}
First we consider the case where $S$ contains all places above primes less than $14$.
Let $\sigma\in\R$ be the abscissa of convergence of the Dirichlet series $\sD(D, \cE_S, \phi, s)$.
Since the infinite product \eqref{eq:de1} diverges at $s=1$, we have $\sigma\geq 1$.
Suppose that $\sigma=1$. 
By Lemma \ref{lem:propzeta} and \cite{Nark}*{Theorem II in p.465}, $\sD(D, \cE_S, \phi, s)$ has a pole at $s=1$. 
This contradicts Lemma \ref{lem:analytic} (2).
Therefore, we obtain $\sigma>1$.
Since the possible poles of $\xi(D, \cE_S, \phi, s)$ are located at $s=\tfrac12$ and $s=\tfrac{3}{2}$,  and since $c(D, \phi, S, s)$ is holomorphic in the region $\Re(s)>1$, we see that $\sigma=\frac32$ and $\sD(D, \cE_S, \phi, s)$ has a simple pole at $s=\tfrac{3}{2}$.
The positivity of the residue is obvious since we know from Lemma \ref{lem:propzeta} that $(s-\tfrac32)\sD(D, \cE_S, \phi, s)>0$ for $s\in\R$, $s>\frac32$.

Now we remove the condition that $S$ contains all places above primes less than $14$.
Take a finite subset $T$ of $\Sigma$ containing $S$. 
By Lemma \ref{lem:unchange}, we have
    \begin{equation}\label{eq:20200810}
    \sD(D, \cE_S, \phi, s)
    =\sum_{\substack{\cE'\in X(D_{T\setminus S}) \\
    \cE'=(\cE'_v)_{v\in T\setminus S}}} 
    \left(\prod_{v\in T\setminus S} 
    \frac{\cD_{\cE'_v}(\pi_v,s)}{N(\ff_{\cE'_v})^{s-1}}\right)
    \, \sD(D,  \cE_S\cup\cE',  \phi,  s).
    \end{equation}
Suppose that $T$ contains all places above primes less than $14$.
Since $\mathcal{D}_{\cE'_v}(\pi_v, \tfrac32)>0$ for $v\notin S$ and $\cE_v'\in X(D_v)$, the assertions follow from those for $\sD(D,  \cE_S\cup\cE',  \phi,  s)$.
\end{proof}

For the rest of this subsection, we assume $\sD(D, \phi, \cE_S, s)$ is not identically zero and compute the residue of $\sD(D, \phi, \cE_S, s)$ at $s=\frac32$.
We set
    \[
    I_{\cE_v}(\Phi_v, \phi_v, s)=c_v\int_{D_v^\times}\int_{D_v}  \Phi_{v,K_v}(x, y) \, |\det(x)|_v^{s-2} \,
    \frac{\langle \phi_v, \pi_v(x)\phi_v \rangle_v}{\langle \phi_v, \phi_v \rangle_v} \,
    \d y \d x,
    \]
where $\Phi_{v,K_v}(x, y)\coloneqq\int_{K_v} \Phi_v((x,y)\cdot \rho(1,1,k)) \, \d k$.
Note that $|\det(x)|_v^{-2}\d x$ is a Haar measure on $D_v^\times$.
This integral $I_{\cE_v}(\Phi_v, \phi_v, s)$ is regarded as a local Godement-Jacquet integral.
Hence it has meromorphic continuation to the whole complex plane and holomorphic at $s=1$.
In particular, for $v\not\in S$ we have 
\begin{multline*}
I_{\cE_v}(\Phi_v, \phi_v, s)=\zeta_{F_v}(2)^{-1}\int_{M_2(\fo_v)\cap \GL_2(F_v)} |\det(x)|_v^{s-2} \,
    \frac{\langle \phi_v, \pi_v(x)\phi_v \rangle_v}{\langle \phi_v, \phi_v \rangle_v} \, \d x \\
    =\zeta_{F_v}(2)^{-1}L(s-\tfrac12, \pi_v),    
\end{multline*}
since $\vol(K_v, \d x)=c_v^{-1}\zeta_{F_v}(2)^{-1}$. 

For $v\in S$, we take $\Phi_v\in\cS(V(F_v))$ satisfying \eqref{eq:cond} and
    \begin{equation}\label{eq:cond2}
    Z_{\cE_v}(\Phi_v, \phi_v, \tfrac32)\neq 0.
     \end{equation}
This is possible by Lemma \ref{lem:loc} (2).
We set 
    \[
    \fP_v(D, \cE_v, \phi_v)
    \coloneqq \frac{\alpha_{\cE_v}(\phi_v, \phi_v)}{|2|_v\zeta_{F_v}(2) \, L(\tfrac{1}{2},\pi_v)} \,
    \frac{I_{\cE_v}(\Phi_v, \phi_v, 1)}{Z_{\cE_v}(\Phi_v, \phi_v, \frac32)}
    \]
and $\fP_S(D,\cE_S, \phi_S) \coloneqq \prod_{v\in S}  \fP_v(D,\cE_v, \phi_v)$.

\begin{lem}\label{lem:20200814}
Assume that $\Phi_v$ satisfies \eqref{eq:cond} and \eqref{eq:cond2} for $v\in S$.
If $\xi(D,  \phi,  \cE_S,  s)$ is not identically zero, then
    \[
    \underset{s=\frac32}{\Res}\, \xi(D, \cE_S, \phi, s)
    =\frac{\fc_F \,\zeta_F(2) \, \langle \phi,\phi\rangle\, L(\tfrac{1}{2},\pi)}
    {2|\Delta_F|^4\, \zeta_F^S(2)^2} 
     \cdot\frac{\fP_S(D, \cE_S, \phi_S)}{\, \alpha_{\cE_S}(\phi, \phi)}.
    \]
In particular, $\fP_v(D,  \cE_v,  \phi_v)$ does not depend on the choice of $\Phi_v$ satisfying \eqref{eq:cond} and \eqref{eq:cond2}.
\end{lem}

\begin{proof}
Recall that $\Phi_v$ is the characteristic function of $V(\fo_v)$ for all $v\notin S$.
From Theorem \ref{thm:global} (2) and \eqref{eq:gloloc},  we can deduce
     \[
     \underset{s=\frac32}{\Res}\, \xi(D, \cE_S, \phi, s)
     = \frac{c_F\, \fc_F}{2}\, Z^\GJ((\sF\Phi_K)_1, \overline{f_\phi},1)
     \left(\prod_{v\in S} Z_{\cE_v}(\Phi_v, \phi_v, \tfrac32)\right)^{-1}.
     \]
By the functional equation of the Godement-Jacquet integral, we get 
    \[
    Z^\GJ((\sF\Phi_K)_1, \overline{f_\phi}, 2-s)
    =Z^\GJ(\sF((\sF\Phi_K)_1),f_{\phi},s)=Z^\GJ((\sF_2(\Phi_K))_1, f_\phi, s)
    \]
Here, we used the partial Fourier inversion formula $\sF((\sF\Phi)_1)(x)=(\sF_2\Phi)_1(-x)$ and$(\sF_2(\Phi_K))_1(-x)=(\sF_2(\Phi_K))_1(x)$, where $\sF_2\Phi$ is the partial Fourier transform:
    \[
    \sF_2\Phi((x, y))=\int_{D_\A}\Phi((x, z))\psi_F(\Tr(yz)) \d z.
    \]
The last expression becomes
    \begin{align*}
    Z^\GJ((\sF_2(\Phi_K))_1, f_\phi, s)=
    &\int_{D_\A^\times} \left(\int_{D_\A} \Phi_K(x, y)\, \d y\right)
    \, f_{\phi}(x) \, |\det(x)|_v^s \d^\times x\\
    &=\frac{\langle \phi,\phi\rangle}{c_F\, |\Delta_F|^4}
    \prod_v I_{\cE_v}(\Phi_v,\phi_v,s) \\
    &=\frac{\langle \phi,\phi\rangle}{c_F\, |\Delta_F|^4\zeta^S_F(2)} \, 
    L^S(s-\tfrac{1}{2},\pi) \, \prod_{v\in S} I_{\cE_v}(\Phi_v,\phi_v,s).
    \end{align*}
Here,  we used $\d y= |\Delta_F|^{-2}\prod_v \d y_v$ and $\d^\times x=c_F^{-1}|\Delta_F|^{-2}\prod_v c_v|\det (x_v)|_v^{-2}\d x_v$.
This completes the proof.
\end{proof}

\begin{thm}\label{thm:re}
Suppose that $\sD(s,D,\phi,\cE_S)$ is not identically zero.
Then,
    \begin{equation}\label{eq:re}
    \fC(D, \cE_S, \phi)= \fc_F \, c_F \, \zeta_F(2) \, \langle \phi,\phi\rangle \,
    \fP_S(D, \cE_S, \phi_S)  \, \frac{L(\tfrac{1}{2},\pi)}{L^S(2,\pi,\mathrm{Ad})}.
     \end{equation}
In addition, we see that $\fP_S(D, \cE_S, \phi_S)>0$.
\end{thm}
\begin{proof}
By Lemma \ref{lem:dir}, $\xi(D, \cE_S, \phi, s)$ has a simple pole at $s=\frac32$.
The theorem follows from Lemma \ref{lem:20200814} and the definition of $\xi(D, \cE_S, \phi, s)$.
The last assertion follows from the positivity of the left hand side, $L(\frac12,\pi)\ge 0$ (see \cite{Guo}), and $L(2,\pi_v,\mathrm{Ad})>0$ for $v\notin S$, which is easily seen from Theorem \ref{thm:BB}.
\end{proof}

\begin{cor}\label{cor:nonv3}
Suppose that $\sD(D, \cE_S, \phi, s)$ is not identically zero.
Let $T$ be a finite subset of $\Sigma$ containing $S$.
For $\cE'=(\cE'_v)_{v\in T\setminus S}\in X(D_{T\setminus S})$,  let $\cE_S\cup\cE'$ be the join of $\cE_S$ and $\cE'$,  which is an element of $X(D_T)=\prod_{v\in T}X(D_v)$.
Then $\sD(D, \cE_S\cup\cE', \phi, s)$ is not identically zero for any $\cE'$.
\end{cor}
\begin{proof}
From \eqref{eq:20200810} we get
    \[
    \fC(D, \cE_S, \phi)
    =\sum_{\cE'\in X(D_{T\setminus S})} 
    \left(\prod_{v\in T\setminus S} 
    \frac{D_{\cE'_v}(\pi_v,\frac32)}{N(\ff_{\cE'_v})^{\frac12}}\right) 
    \fC(D,  \cE_S\cup\cE',  \phi).
    \]
Substituting \eqref{eq:re} and multiplying the both sides by $L^T(2,\pi,\mathrm{Ad})\fP_S(D, \cE_S, \phi_S)^{-1}$, we obtain
    \begin{equation}\label{eq:20210118}
    \prod_{v\in T\setminus S}L(2,\pi_v,\mathrm{Ad})^{-1}
    =\sum_{\substack{\cE'\in X(D_{T\setminus S}), \\ 
    \sD(D,  \cE_S\cup\cE',  \phi,  s)\not\equiv0}}
    \prod_{v\in T\setminus S}\mathfrak{R}(\cE'_v, \pi_v),
    \end{equation}
where we write
    \[
    \mathfrak{R}(\cE'_v, \pi_v)= 
    \frac{\cD_{\cE'_{v}}(\pi_{v}, \frac32)\,  \fP_v(D, \cE'_v, \phi_v)}
    {N(\ff_{\cE'_{v}})^{\frac12}}.
    \]
    
Note that we will see in Corollary \ref{cor:fPun} that $\fP_v(D,  \cE_v' ,\phi_v)>0$ for $v\notin S$, since we have chosen the element $d_{\cE_v'}$ satisfying the assumption of Corollary \ref{cor:fPun}. 
Hence we get $\mathfrak{R}(\cE_v', \pi_v)>0$ for $v\notin S$ and the right hand side of \eqref{eq:20210118} has an obvious bound
    \begin{align}\label{eq:202101182}
    \sum_{\substack{\cE'\in X(D_{T\setminus S}), \\ 
    \sD(D,  \cE_S\cup\cE',  \phi, s)\not\equiv0}}
    \prod_{v\in T\setminus S}\mathfrak{R}(\cE'_v, \pi_v) \nonumber
    &\leq\sum_{\cE'\in X(D_{T\setminus S})} 
    \prod_{v\in T\setminus S}\mathfrak{R}(\cE'_v,  \pi_v) \\ 
    &= \prod_{v\in T\setminus S} 
    \left(\sum_{\cE'_v\in X(D_v)}\mathfrak{R}(\cE'_v, \pi_v)\right).
    \end{align}
The equality holds if and only if $\sD(D,  \cE_S\cup\cE',  \phi, s)\not\equiv0$ for all $\cE'\in X(D_{T\setminus S})$.

We compute $\sum_{\cE'_v\in X(D_v)}\mathfrak{R}(\cE'_v, \pi_v)$ for $v\not\in S$.
Recall that $\cD_{\cE'_v}(\pi_v,  \tfrac32)$ equals
    \[
    \begin{cases} 
        (1+q_v^{-1})(1+q_v^{-2})(1-q_v^{-1}+q_v^{-2}) -2q_v^{-3}\lambda_v 
        & \text{if $\cE'_v$ is split}, \\ 
        (1+q_v^{-1})(1+q_v^{-2})(1-q_v^{-1}+q_v^{-2})+2q_v^{-3}\lambda_v
        & \text{if $\cE'_v/F_v$ is an unramified extension}, \\
        1+q_v^{-2}& \text{if $\cE'_v/F_v$ is a ramified extension}.
    \end{cases}
    \]
On the other hand,  we will show in Corollary \ref{cor:fPun} that $\fP_v(D, \cE'_v,  \phi_v)$ equals
    \[
    \frac{1-q_v^{-1}}{2(1+q_v^{-1})}\cdot
    \begin{cases}
    1+q_v^{-1}+q_v^{-1}\lambda_v & \text{if $\cE_v'$ is split}, \\
    1+q_v^{-1}-q_v^{-1}\lambda_v & 
    \text{if $\cE_v'/F_v$ is an unramified extension}, \\
    q_v^{-\frac12}\{(1+q_v^{-1})^2-q_v^{-2}\lambda_v^2\} &
    \text{if $\cE_v'/F_v$ is a ramified extension}.
    \end{cases}
    \]
Thus we get
    \begin{align*}
    \sum_{\cE'_v\in X(D_v)}\mathfrak{R}(\cE'_v, \pi_v) 
    &=(1-q_v^{-2})\{(1+q_v^{-2})^2-q_v^{-3}\lambda_v^2\} \\
    &=(1-q_v^{-2})(1-\alpha_v^2q_v^{-2})(1-\alpha_v^{-2}q_v^{-2}) 
    =L(2,\pi_v,\mathrm{Ad})^{-1}.
    \end{align*}
Thus the equality holds in \eqref{eq:202101182}.
\end{proof}

\begin{rem}
There is a different proof of Corollary \ref{cor:nonv3} using the results of \cite{FH}, \cite{Tunnell} and \cite{Saito}.
\end{rem}


\subsection{The local factor $\fP_v(D, \cE_v, \phi_v)$}\label{sec:localdensities}

In this subsection, we fix a place $v\in\Sigma$ and compute the local factor $\fP(D_v, \cE_v, \phi_v)$ for $\cE_v\in X(D_v)$.

First, we introduce a meromorphic function
    \[
    \fP_v(D, \cE_v, \phi_v, s)
    =\frac{|4|_v^{s-2}\alpha_{\cE_v}(\phi_v, \phi_v)}{\zeta_{F_v}(2)L(2s-\tfrac52, \pi_v)}
    \frac{I_{\cE_v}(\Phi_v, \phi_v, 2s-2)}{Z_{\cE_v}(\Phi_v, \phi_v, s)}.
    \]
Note that $\fP_v(D, \cE_v, \phi_v)=\fP_v(D, \cE_v, \phi_v, \tfrac32)$.
We also consider an auxiliary integral
    \[
    \cI_{\cE_v}(\phi_v, s)
    =\zeta_{F_v}(2)^{-1}\int_{k=\left(
        \begin{smallmatrix}
        a&c \\
        b&d
        \end{smallmatrix}\right)\in K_v} |\det(a+b\delta_{\cE_v})|_v^{s-2} 
     \frac{\langle \phi_v, \pi_v(a+b\delta_{\cE_v})\phi_v \rangle_v}
     {\langle \phi_v, \phi_v \rangle_v}\d k.
    \]
When $\cE_v$ is a field, this integral converges absolutely for all $s\in\C$ and defines an entire function.

\begin{lem}\label{lem:fP}
Suppose that $v$ is a finite place.
\begin{itemize}
\item[(1)] The integral for $\cI_{\cE_v}(\phi_v, s)$ converges absolutely for $\Re(s)$ sufficiently large and has meromorphic continuation to the whole complex plane.
\item[(2)] We have
    \[
    \fP_v(D, \cE_v, \phi_v, s)
    =\frac{L(1, \eta_{\cE_v})^2}{2|d_{\cE_v}|_v^{s-2}L(2s-\frac52, \pi_v)} \,
    \cI_{\cE_v}(\phi_v, 2s-2).
    \]
\end{itemize}
\end{lem}

\begin{proof}
(1) Let $\pr_1\,\colon V(F_v)\rightarrow D_v$ be the projection onto the first coordinate.
Take a sufficiently small open compact subgroup $K_v'$ of $D_v^\times$ so that $K'_v$ fixes $\phi_v$ and $U_v=x_{\cE_v}\cdot\rho(K_v'\times K_v'\times K_v)$ is an open compact neighborhood of $x_{\cE_v}$ contained in $V_{\cE_v}(F_v)$.
Note that $\pr_1(U_v)\subset D_v^\times$.
Let $\Phi_v$ be the characteristic function of $U_v$.
Since $V(F_v)\setminus (D_v^\times\times D_v)$ is of measure $0$,  the domain of integration for the defining integral of $I_{\cE_v}(\Phi_v, \phi_v, s)$ can be replaced with $V(F_v)$. 
Hence, by change of variable for $(x,y)\in V(F_v)$ we have
    \begin{align*}
    & \langle \phi_v, \phi_v \rangle_v \cdot I_{\cE_v}(\Phi_v, \phi_v, s) \\
    &= c_v\int_{K_v}\int_{V(F_v)} \Phi_v(x,y)\, |\det(\mathrm{pr}_1((x,y)\cdot k))|_v^{s-2}    \langle \phi_v, \pi_v(\mathrm{pr}_1((x,y)\cdot k))\phi_v \rangle_v   \d x \d y \d k\\
    &= c_v \int_{K_v}\int_{U_v}  |\det(\mathrm{pr}_1(x_{\cE_v}\cdot k))|_v^{s-2}  \langle \phi_v, \pi_v(\mathrm{pr}_1(x_{\cE_v}\cdot k))\phi_v \rangle_v  \d x \d y \d k\\
    &= c_v \, \vol(U_v)  \int_{k=\left(\begin{smallmatrix}a&c \\ b&d \end{smallmatrix}\right)\in K_v} |\det(a+b\delta_{\cE_v})|_v^{s-2}  \langle \phi_v, \pi_v(a+b\delta_{\cE_v})\phi_v \rangle_v  \d k\\
    &= c_v \, \zeta_{F_v}(2)\, \vol(U_v)\,  \langle \phi_v, \phi_v \rangle_v \cdot \cI_{\cE_v}(\phi_v, s),
    \end{align*}
where we abbreviate $(x,y)\cdot \rho(1,1,k)$ as $(x,y)\cdot k$. 
In addition,  the absolute convergence and meromorphic continuation follow from those for the Godement-Jacquet integral $I_{\cE_v}(\Phi_v, \phi_v, s)$.

(2) Let $\Phi_v$ be as above.
This is an immediate consequence of 
    \[
    Z_{\cE_v}(\Phi_v, \phi_v, s)
    =\frac{2c_v}{L(1, \eta_{\cE_v})^2}
    \, \vol(U_v)|P(x_{\cE_v})|_v^{s-2}\alpha_{\cE_v}(\phi_v, \phi_v)
    \]
and $P(x_{\cE_v})=4d_{\cE_v}$.
\end{proof}

\begin{lem}\label{lem:omg}
Suppose that $v$ is a finite place.
The integral for $\cI_{\cE_v}(\phi_v, s)$ converges absolutely and uniformly for $\Re(s)\geq 1$ and satisfies
    \[
    \cI_{\cE_v}(\phi_v, 1)
    =c_v^{-1} L(1, \eta_{\cE_v})^{-1} \,
    \frac{\alpha_{\cE_v}(\phi_v, \phi_v)}{\langle \phi_v, \phi_v \rangle_v}.
    \]
\end{lem}

\begin{proof}
Let $ U_v$ and $\Phi_v$ be as in the proof of Lemma \ref{lem:fP}. 
Note that we have
\[
\int_{k=\left(\begin{smallmatrix}a&c \\ b&d \end{smallmatrix}\right)\in K_v} f(a,b) \d k = \zeta_{F_v}(2) \int_{(\fo_v\times\fo_v)\setminus(\varpi_v\fo_v\times\varpi_v\fo_v)} f(a,b) \d a \d b 
\]
for $f\in\Cc(F_v\times F_v)$, where $\d a$ and $\d b$ are Haar measures on $F_v$ normalized so that $\vol(\fo_v)=1$.
Formally we have
    \begin{align*}
    &\hspace{-25pt}c_v\int_{V(F_v)}\left|\Phi_{v, K_v}(x, y)|\det(x)|_v^{-1}
    \frac{\langle \phi_v, \pi_v(x)\phi_v\rangle_v}{\langle \phi_v, \phi_v \rangle_v}
    \right| \d x \d y \\
    &=c_v\zeta_{F_v}(2)\vol(U_v)
    \int_{(\fo_v\times\fo_v)\setminus(\varpi_v\fo_v\times\varpi_v\fo_v)}
    \frac{|\langle \phi_v, \pi_v(a+b\delta_{\cE_v})\phi_v\rangle_v|}
    {\langle \phi_v, \phi_v\rangle_v|\det(a+b\delta_{\cE_v})|_v}\d a \d b \\
    &=c_v\zeta_{F_v}(2)\vol(U_v)\left(\int_{\fo_v\times\fo_v^\times}
    +\int_{\fo_v^\times\times\varpi_v\fo_v}\right)
    \frac{|\langle \phi_v, \pi_v(a+b\delta_{\cE_v})\phi_v\rangle_v|}
    {\langle \phi_v, \phi_v\rangle_v|a^2-b^2d_{\cE_v}|_v}\d a \d b.
    \end{align*}
Since the central character of $\pi_v$ is trivial, the integral over $\fo_v\times\fo_v^\times$ equals
    \[
    \vol(\fo_v^\times)\int_{\fo_v}
    \frac{|\langle \phi_v, \pi_v(a+\delta_{\cE_v})\phi_v\rangle_v|}
    {\langle \phi_v, \phi_v\rangle_v|a^2-d_{\cE_v}|_v}\d a.
    \]
Making the substitution $c=b^{-1}$, the integral over $\fo_v^\times\times\varpi_v\fo_v$ becomes
\begin{multline*}
    \vol(\fo_v^\times)\sum_{l=1}^\infty\int_{\varpi_v^l\fo_v^\times}
    \frac{|\langle \phi_v, \pi_v(1+b\delta_{\cE_v})\phi_v\rangle_v|}
    {\langle \phi_v, \phi_v\rangle_v|1-b^2d_{\cE_v}|_v}\d b  \\
    =\vol(\fo_v^\times)\sum_{l=1}^\infty\int_{\varpi_v^{-l}\fo_v^\times}
    \frac{|\langle \phi_v, \pi_v(c+\delta_{\cE_v})\phi_v\rangle_v|}
    {\langle \phi_v, \phi_v\rangle_v|c^2-d_{\cE_v}|_v}\d c.    
\end{multline*}
Hence, using
\[
c_v \, L(1,\eta_{\cE_v}) \, \vol(\fo_v^\times) \, \int_{F_v^\times} f(a+\delta_{\cE_v}) \frac{\d a}{|a^2-\delta_{\cE_v}|_v}=\int_{F_v^\times\bs \cE_v^\times}f(h_v)\, \d h_v 
\]
($f$ is a test funciton on $F_v^\times\bs \cE_v^\times$), at least formally we have
    \begin{align*}
    &\hspace{-45pt}c_v\int_{V(F_v)}\left|\Phi_{v, K_v}(x, y)|\det(x)|_v^{-1}
    \frac{\langle \phi_v, \pi_v(x)\phi_v\rangle_v}{\langle \phi_v, \phi_v \rangle_v}
    \right| \d x \d y \\
    &=\zeta_{F_v}( 2)L(1, \eta_{\cE_v})^{-1}\vol(U_v)\int_{F_v^\times\bs\cE_v^\times}
    \frac{|\langle \phi_v, \pi_v(h_v)\phi_v\rangle_v|}{\langle \phi_v, \phi_v\rangle_v}
    \d h_v.
    \end{align*}
Since the right hand side converges by \cite{Wal2}*{Lemmas 2, 3}, the integral for $I_{\cE_v}(\Phi_v, \phi_v, s)$ converges absolutely at $s=1$.

Let $U_1$ be the set of $(x, y)\in V(F_v)$ such that $|\det(x)|_v<1$. 
Suppose $\Re(s)\ge 1$, that is, $\Re(s)-2\ge -1$. 
Then, we have $|\det(x)|_v^{\Re(s)-2}\leq |\det(x)|_v^{-1}$ for $(x, y)\in U_1$. 
Hence the convergence of $I_{\cE_v}(\Phi_v, \phi_v, s)$ follows from that of $I_{\cE_v}(\Phi_v, \phi_v, 1)$.
Therefore the integral for $I_{\cE_v}(\Phi_v, \phi_v, s)$ converges absolutely and uniformly for $\Re(s)\geq1$, and so is $ \cI_{\cE_v}(\phi_v, s)$.
Now all the formal manipulations are justified and we obtain the desired equality.
\end{proof}

Next we consider the archimedean case.  
First,  we prepare some general lemmas about norms on real vector spaces.

\begin{lem}\label{lem:ineq1}
Let $\|\cdot\|$ be the norm on $\R^n$ given by $\| x\|=(\t xx)^\frac12$ for $x\in\R^n$,  where elements in $\R^n$ are regarded as column vectors.
We consider the natural left action of $\GL_n(\R)$ on $\R^n$.
For a compact subset $C\subset \GL_n(\R)$,  we have
    \[
    \| g\cdot x \| \ll \| x \|,  \qquad g\in C,  \ x\in\R^n.
    \]
\end{lem}

\begin{proof}
Let $\Sym_n(\R)$ be the set of real symmetric matrices of size $n$ and $\Omega$ be the subset of positive definite matrices.
For $x,  y\in \Sym_n(\R)$,  we write $y<x$ if $x-y$ is positive definite.

We have a surjective map $h\,\colon \GL_n(\R)\to\Omega$ given by $h(g)=\t gg$ for $g\in\GL_n(\R)$.
Since it is continuous,  $h(C)$ is a compact subset of $\Omega$.
For $g\in C$,  let $\alpha_1(g),  \ldots,  \alpha_n(g)>0$ be the eigenvalues of $h(g)$.
Since $\Tr(h(C))$ is a compact subset of $\R_{>0}$,  there is $c\in\R_{>0}$ such that $\Tr(h(g))<c$ for all $g\in C$.
In particular,  we have $\alpha_j(g)<c$ for all $j$ and hence $h(g)<cI_n$ for all $g\in C$.

Thus for $x\in\R^n$,  
    \[
    \| g\cdot x \|^2=\t x\t ggx=\t x\, h(g)\, x < c\t x x= c \| x\|^2.  \qedhere
    \]
\end{proof}

\begin{lem}\label{lem:ineq2}
Let $\|\cdot\|'$ be the norm on $\GL_2(\R)$ given by $\| g\|'=|\det(g)^{-1}\Tr(g\t g)|^\frac12$ for $g\in\GL_2(\R)$.
For a compact subset $C\subset \GL_2(\R)$,  we have
    \[
    \| gx \|' \ll \| x \|',  \qquad g\in C,  \ x\in\GL_2(\R).
    \]
\end{lem}

\begin{proof}
We define the action of $\GL_2(\R)$ on $\Sym_2(\R)$ by $X\cdot g=(\det(g))^{-1}\t gXg$ for $g\in\GL_2(\R)$ and $X\in\Sym_2(\R)$.
Note that this action factors through that of $\PGL_2(\R)$ and gives rise to a continuous homomorphism
    \[
    f\,\colon \GL_2(\R)\to \GL(\Sym_2(\R))\simeq\GL_3(\R).
    \]
Let $\|\cdot\|$ denote the norm on $\M_3(\R)$ given by $\|x\|=\Tr(x\t x)^\frac12$ for $x\in\M_3(\R)$.
Note that this norm is equivalent to the one considered in Lemma \ref{lem:ineq1}.

By the Cartan decomposition, every element $g\in\GL_2(\R)$ can be written in the form $g=k_1\diag(ab, b)k_2$ with $k_1,  k_2\in\O(2)$, $a>0$, and $b>0$.
Then $\| g \|'=\|\diag(a,  1)\|'=(a+a^{-1})^\frac12$ and $\|f(g)\|=\|\diag(a,  1,  a^{-1})\|=(a^2+1+a^{-2})^\frac12$.
Hence we have $\| g \|' \asymp \| f(g) \|^\frac12$ for $g\in\GL_2(\R)$.

From this,  we get $\|gx\|' \ll \|f(gx)\|^\frac12$ and $\|f(x)\|^\frac12 \ll \|x\|'$ for $g,  x\in\GL_2(\R)$.
Since we have $\|f(gx)\| = \|f(g)f(x)\| \ll \|f(x)\|$ for $g\in C$ and $x\in\GL_2(\R)$ from Lemma \ref{lem:ineq1},  the assertion follows.
\end{proof}

We define a norm $\|\cdot\|$ on $V(F_v)$ by
    \[
    \|z\|=\Tr(xx^\ast+yy^\ast)^\frac12, \hspace{25pt} z=(x, y)\in V(F_v),
    \]
where 
    \[
    x^\ast=
        \begin{cases}
        \t x & \text{if $D_v=\M_2(\R)$}, \\ 
        \t\bar{x} & \text{if $D_v=\M_2(\C)$}, \\
         x^\iota & \text{if $D_v$ is the quaternion division algebra}.
         \end{cases} 
    \] 
Take $\varepsilon>0$ and let $\psi_\varepsilon\in\Cc(V(F_v))$ be a non-negative $K_v$-invariant function such that 
    \[
    \int_{V(F_v)}\psi_\varepsilon(x)\d x=\varepsilon^8
    \]
and supported on $B_\varepsilon=\{x\in V(F_v) \mid \| x\|<\varepsilon\}$.
Let $\Phi_\varepsilon(x)=\psi_\varepsilon(x-x_{\cE_v})$.

\begin{lem}\label{lem:convarch}
Suppose that $v$ is an archimedean place.
\begin{itemize}
\item[(1)] 
For a $K_v$-finite vector $\phi_v\in\pi_v$, the integral for $I_{\cE_v}(\Phi_v, \phi_v, s)$ converges absolutely and uniformly on compact sets for $\Re(s)\geq1$.

\item[(2)] Suppose that $\varepsilon>0$ is sufficiently small.
For $\Phi_v=\Phi_\varepsilon$, we have 
    \[
    \lim_{\varepsilon\to0}\ \varepsilon^{-8} I_{\cE_v}(\Phi_v, \phi_v, 1)
    =\zeta_{F_v}(2)\, \cI_{\cE_v}(\phi_v, 1).
    \]

\item[(3)]  We have
    \[
    \fP_v(D, \cE_v, \phi_v)
    =\frac{L(1, \eta_{\cE_v})^2|d_{\cE_v}|_v^\frac12}{2L(\frac12, \pi_v)}
    \cI_{\cE_v}(\phi_v, 1).
    \]
\end{itemize}
\end{lem}

\begin{proof}
(1) 
Suppose that $D_v$ is not split, that is, $\det$ is positive definite on $D_v$.
Set $D_v^1\coloneqq\{x\in D_v \mid \det(x)=1 \}$. 
Consider the polar decomposition $x=r \, x_1$ for $x\in D_v^\times$, $r\in\R_{>0}$, $x_1\in D_v^1$, and choose a measure $\d x_1$ on $D_v^1$ so that $\d x=r^3\, \d r \, \d^1 x$. 
Then 
\[
I_{\cE_v}(\Phi_v, \phi_v, s)=\int_{\R_{>0}} r^{2s} \, \tilde\Phi(r) \, \frac{\d r}{r} ,
\]
\[
\tilde\Phi(r)=c_v\int_{D_v^1} \int_{D_v} \Phi_{v,K_v}(rx_1,y) \frac{\langle \phi_v,\pi_v(x_1)\phi_v\rangle}{\langle \phi_v,\phi_v\rangle}\, \d x_1 \, \d y . 
\]
The function $\tilde\Phi$ belongs to $\cS(\R)$, because $D_v^1$ is compact.  
Therefore, we obtain the assertion in this case. 

Next, we consider the case where $D_v$ is split. 
By the same argument as in the proof of Lemma \ref{lem:omg},  it suffices to show the absolute convergence at $s=1$.
By the Cartan decomposition,  any $g\in\GL_2(F_v)$ can be written in the form $g=k_1\diag(a, b)k_2$ with some $k_1, k_2\in K_v$ and $a>b>0$.
Then,  a Haar measure $\d g$ on $\GL_2(F_v)$ is given by
\[
|\det(g)|_v^2\, \d g=C'(a^2-b^2)^{\dim_\R F_v}(ab)^{\dim_\R F_v-1}\d a \d b\d k_1\d k_2
\]
with some constant $C'>0$. 
See \cite{Helgason}*{Theorem 5.8 in p.186} for details.
On the open dense subset $\GL_2(F_v)$ of $M_2(F_v)$, a Haar measure $\d x$ on $M_2(F_v)$ equals $|\det(g)|_v^2\, \d g$ up to a constant, since the complement $M_2(F_v)\setminus \GL_2(F_v)$ is of measure $0$ with respect to $\d x$. 
Hence there is $C>0$ such that
    \begin{align*}
    &\hspace{-8pt}\int_{V(F_v)}\left|\Phi_{v, K_v}(x, y)|\det(x)|_v^{-1}
    \frac{\langle \phi_v, \pi_v(x)\phi_v \rangle_v}{\langle \phi_v, \phi_v\rangle_v}\right| 
    \d x\d y  \\
    &=C\int_{a>b>0}\d a \d b \int_{K_v}\d k_1 \int_{K_v}\d k_2 \int_{\M_2(F_v)}\d y \\
    &\hspace{40pt}(ab)^{-1}(a^2-b^2)^{\dim_\R F_v}|\Phi_{v, K_v}(k_1\diag(a, b)k_2, y)|\left|
        \frac{\langle \phi_v, \pi_v(k_1\diag(a, b)k_2)\phi_v \rangle_v}
        {\langle \phi_v, \phi_v\rangle_v}\right|. 
    \end{align*}
Since $\phi_v$ is $K_v$-finite and $\Phi_{v, K_v}$ is compactly supported, there are $K_v$-finite vectors $\phi_{v, j}\in\pi_v$ and $\Psi_j\in\Cc(\R\times\R)$, $j=1, 2,\ldots, l$ such that the above integral is bounded by
    \[
    C\sum_{j=1}^l \int_{a>b>0}(a^2-b^2)^{\dim_\R F_v}\Psi_j(a, b)\left|
    \frac{\langle \phi_{v, j}, \pi_v(\diag(a, b))\phi_{v, j} \rangle_v}{\langle \phi_v, \phi_v\rangle_v}\right|
    \frac{\d a}{a} \frac{\d b}{b}.
    \]
Making the substitution $c=ab^{-1}$, this becomes
    \[
    C\sum_{j=1}^l \int_{c>1}\int_{b>0}(b^2c^2-b^2)^{\dim_\R F_v}\Psi_j(bc, b)\left|
    \frac{\langle \phi_{v, j}, \pi_v(\diag(c, 1))\phi_{v, j} \rangle_v}
    {\langle \phi_v, \phi_v\rangle_v}\right|
    \frac{\d c}{c} \frac{\d b}{b}.
    \]
Using
    \begin{equation}\label{eq:mcest}
    |\langle \phi_1, \pi_v(\diag(c, 1))\phi_2 \rangle_v|\ll c^{-\kappa},  \qquad
    \phi_1,  \phi_2\in\pi_v
    \end{equation}
for some $\kappa>0$, we see that the above integral is bounded.
If $\pi_v$ is a discrete series representation, \eqref{eq:mcest} is well-known.
If $\pi_v$ is a unitary principal series representation, \eqref{eq:mcest} follows from \cite{Knapp}*{Proposition 7.14, 7.15}.
This completes the proof.

\noindent

(2) 
Let $\pr_1\,\colon V(F_v)\rightarrow D_v$ denote the projection onto the first coordinate. 
Set $U=\{ (x,k)\in V(F_v)\times K_v \mid \det(\pr_1((x+x_{\cE_v})\cdot\rho(1, 1, k)))\neq 0\}$. 
Then, $U$ is open dense in $V(F_v)\times K_v$. 
Assume that $\varepsilon$ is sufficiently small so that $\Phi_v=\Phi_\varepsilon$ is supported on the set of $x \in V(F_v)$ such that $(x,  k)\in U $ for some $k\in K$. 
For $(x,k)\in U$, we set
    \[
    f(x, k)=\frac{\langle \phi_v, \pi_v(\pr_1((x+x_{\cE_v})
    \cdot\rho(1, 1, k)))\phi_v\rangle_v}
    {\langle \phi_v, \phi_v \rangle_v |\det(\pr_1((x+x_{\cE_v})\cdot\rho(1, 1, k)))|_v}.
    \]    
Changing the variables,  we obtain
    \begin{equation}\label{eq:fxk}
    \int_{V(F_v)} \int_{K_v} \psi_\varepsilon(x)\, f(x, k) \, \d k \, \d x =I_{\cE_v}(\Phi_v,  \phi_v,  1).        
    \end{equation}
Note that the absolute convergence of the left hand side follows from the first part of this lemma. 

We need the following claim.
    \begin{quote}\textbf{Claim.}
    There exists a small constant $\delta>0$ such that the integral $\int_{K_v} \, f(x, k) \, \d k$ is convergent uniformly on $x\in B_\delta$. 
    \end{quote}

\medskip
\noindent
{\it Proof of Claim.} 
Consider the case $\cE_v=\C$. 
If $\delta$ is sufficiently small, there exists $t>0$ such that $|\det(\pr_1((x+x_{\cE_v})\cdot\rho(1, 1, k)))|_v>t$ for any $x\in B_\delta$ and $k\in K_v$. 
This fact is easily proved by a direct calculation similar to the proof of Lemma \ref{lem:omgarch}.   
The claim is obvious in this case.

In what follows,  we assume $D_v=M_2(F_v)$ and $\cE_v=F_v\times F_v$.
In this case,  in order to show the uniform convergence,  we need to take care of the zeros of the denominator of $f(x,k)$. 

When $F_v=\R$, 
    \[
    \tilde{f}(x-x_{\cE_v})=\text{(constant)}\times \int_{-\pi}^\pi  
    \frac{\langle \phi_v, \pi_v(x_1\cos\theta+x_2\sin\theta)\phi_v\rangle_v}
    {|\det(x_1\cos\theta+x_2\sin\theta)|_v} \d \theta,
    \]
where $x=(x_1,x_2)$. 
Assume that $\delta>0$ is a sufficiently small constant. 
It suffices to show that this integral converges uniformly on $V_{\cE_v,\delta}\coloneqq\{x\in V(F_v)\mid  \|x-x_{\cE_v}\|<\delta\}$. 
Let $x\in V_{\cE_v,\delta}$.
Note that $\det(x_2)\neq 0$, because $x_2$ is close to $\delta_{\cE_v}$. 
Putting $z=\tan\theta$ and $x_3=x_2^{-1}x_1$,  the above integral becomes
    \[
     \int_\R  \frac{\langle \phi_v, \pi_v(z x_2+x_1)\phi_v\rangle_v}
    {|\det(x_2)|_v\, |\det(zI_2+x_3)|_v} \d z.
    \]
    
Since $x\in V_{\cE_v,\delta}$,  we see that $x_3$ is sufficiently close to $\delta_{\cE_v}^{-1}$. 
Set $l_x\coloneqq\tfrac12\Tr(x_3)$.
Changing the variables $z\mapsto z-l_x$,  we have
    \[
    \tilde{f}(x-x_{\cE_v})= \text{(constant)}\times \int_\R  
    \frac{\langle \phi_v, \pi_v((z-l_x)x_2+x_1)\phi_v\rangle_v}
    {|\det(x_2)|_v\, |\det(zI_2+(x_3-l_xI_2))|_v} \d z.
    \]
Since the trace of $x_3-l_xI_2$ is zero, by diagonalizing $2\times 2$ symmetric matrices, we can take $k_x\in K_v$ close to $1$ so that $u_x\coloneqq k_x^{-1}(x_3-l_xI_2)k_x$ is anti-diagonal and $u_x$ is still close to $\delta_{\cE_v}^{-1}$.
Hence $|\det(u_x)|_v$ is close to $|d_{\cE_v}|_v^{-1}$.
Set $\mathfrak{u}_x\coloneqq\sqrt{|\det(u_x)|_v}$.
Note that the diagonalization of $\mathfrak{u}_x zI_2-u_x$ is $\mathfrak{u}_x\diag(z-1,z+1)$.
Again we change the variables $z\mapsto \mathfrak{u}_x z$ to obtain
    \begin{equation}\label{eq:1221}
    \tilde{f}(x-x_{\cE_v})= \text{(constant)}\times \int_\R  
    \frac{\langle \phi_v, \pi_v((\mathfrak{u}_x z-l_x) x_2+x_1)\phi_v\rangle_v}
    {|\det(x_2)|_v\,|\mathfrak{u}_x|_v\, |z-1|_v\, |z+1|_v} \d z. 
    \end{equation}

Set $X=(\mathfrak{u}_x\,z-l_x) x_2+x_1$.
Since $x$ is in a small relatively compact set $V_{\cE_v,\delta}$,   it follows from Lemma \ref{lem:ineq2} that
    \begin{align*}
    \| X \|' & \ll \| zI_2+\mathfrak{u}_x^{-1}(x_3-l_xI_2) \|'=\| \diag(z-1,  z+1) \|'  \\
    & \ll  \left( \left|\frac{z-1}{z+1}\right|_v+\left|\frac{z+1}{z-1}\right|_v \right)^{1/2}.
    \end{align*}
Here,   $\| \; \|'$ is the norm on $\GL_2(\R)$ defined in Lemma \ref{lem:ineq2}. 
Write $X=k_1\diag(c,  1)k_2$ with $k_1,  k_2\in K_v$ and $c>1$.
It follows from \eqref{eq:mcest} that  
    \begin{align}\label{eq:1222} \nonumber
    \langle \phi_v, \pi_v(X)\phi_v\rangle_v 
    & \ll c^{-\kappa} \asymp (c+c^{-1})^{-\frac{\kappa}{2}}=\| X \|'^{-\kappa} \\
    &  \ll  \left( \left|\frac{z-1}{z+1}\right|_v+\left|\frac{z+1}{z-1}\right|_v \right)
    ^{-\frac{\kappa}{2}}
    \ll \min\{|z-1|_v,|z+1|_v\}^{\frac{\kappa}{2}}.
    \end{align}
Here, the implicit constant does not depend on $x\in V_{\cE_v,\delta}$. 
Therefore, to prove the uniform convergence, we have only to consider the integral for \eqref{eq:1221} around $z=1$ and $z=-1$. 
Thus, by \eqref{eq:1221} and \eqref{eq:1222} we get the uniform convergence of the integral for $\tilde{f}(x-x_{\cE_v})$ on $V_{\cE_v,\delta}$.

For the case $F_v=\C$, we can similarly show
    \[
    \tilde{f}(x-x_{\cE_v})= \text{(constant)}\times \int_\C  
    \frac{\langle \phi_v, \pi_v((\mathfrak{u}_x\,z-l_x) x_2+x_1)\phi_v\rangle_v}
    {|\det(x_2)|_v\,|u_1u_2|_v\, |z-1|^2\, |z+1|^2} \d z ,
    \]
    \[
    \langle \phi_v, \pi_v((\mathfrak{u}_x\,z-l_x) x_2+x_1)\phi_v\rangle_v 
    \ll \min\{|z-1|,|z+1|\}^{\kappa},
    \]
by using the calculations as in the proof of Lemma \ref{lem:omgarch}.
This proves the claim we desire. \qed

\medskip
Define $\tilde{f}:B_\delta\to \C$ by $\tilde{f}(x)\coloneqq\int_{K_v} \, f(x, k) \, \d k$. 
As a consequence of the above claim, $\tilde{f}$ is a continuous function on $B_\delta$. 
From \eqref{eq:fxk} we obtain
    \[
    I_{\cE_v}(\Phi_\varepsilon, \phi_v, 1)
    =\int_{V(F_v)}\psi_\varepsilon(x) \, \tilde{f}(x)\,  \d x.
    \]
Set $f_0(x)\coloneqq \tilde{f}(x)-\zeta_{F_v}(2)\cI_{\cE_v}( \phi_v, 1)$. 
By $\varepsilon^{-8}\int_{V(F_v)}\psi_\varepsilon(x) \, \d x=1$, we have
    \[
    \varepsilon^{-8}I_{\cE_v}(\Phi_\varepsilon, \phi_v, 1)
    =\zeta_{F_v}(2)\cI_{\cE_v}(\phi_v, 1)+\varepsilon^{-8}\int_{V(F_v)} \psi_\varepsilon(x) \, f_0(x)\,  \d x .
    \]
Since $\lim_{x\to 0}\tilde{f}(x)=\tilde{f}(0)=\zeta_{F_v}(2)\cI_{\cE_v}( \phi_v, 1)$, we have $\lim_{x\to 0} f_0(x)=0$. 
For arbitrary small $\epsilon>0$, there exists $0<\delta_\epsilon(<\delta)$ so that $|f_0(x)|<\epsilon$ for any $x\in B_{\delta_\epsilon}$. 
Hence, for any $\epsilon>0$, we have
\[
\limsup_{\varepsilon\to 0} \left| \varepsilon^{-8}\int_{V(F_v)} \psi_\varepsilon(x) \, f_0(x)\,  \d x \right| < \epsilon . 
\]
This means $\lim_{\varepsilon\to 0} \varepsilon^{-8}\int_{V(F_v)} \psi_\varepsilon(x) \, f_0(x)\,  \d x=0$, hence
    \[
    \lim_{\varepsilon\to0}\varepsilon^{-8}I_{\cE_v}(\Phi_\varepsilon, \phi_v, 1)
    =\zeta_{F_v}(2)\cI_{\cE_v}(\phi_v, 1).
    \]

(3) Since $|P(x)|_v^{-1/2}$ and $\alpha_{\cE_v}(\pi_v(g_1)\phi_v,\pi_v(g_2)\phi_v)$ are smooth on a small neighborhood of $x_{\cE_v}\in V_{\cE_v}(F_v)$, it follows from the Taylor expansion that $|P(x)|_v^{s-2}=|P(x_{\cE_v})|_v^{s-2}+O(\| x-x_{\cE_v}\|)$ and $\alpha_{\cE_v}(\pi_v(g_1)\phi_v,\pi_v(g_2)\phi_v)=\alpha_{\cE_v}(\phi_v,\phi_v)+O(\| x-x_{\cE_v}\|)$, where $x=x_{\cE_v}\cdot \rho(g_1,g_2,g_3)$.
Thus we get
    \[
    Z_{\cE_v}(\Phi_\varepsilon, \phi_v, s)
    =\frac{2\, \varepsilon^8}{L(1, \eta_{\cE_v})^2}
    \left(|P(x_{\cE_v})|_v^{s-2}\alpha_{\cE_v}(\phi_v, \phi_v)+O(\varepsilon)\right).
    \]
This implies that \eqref{eq:cond2} holds by the assumption $\alpha_{\cE_v}(\phi_v, \phi_v)\neq 0$.
Note that \eqref{eq:cond} holds by the definition of $\Phi_\varepsilon$. 
Since $\fP_v(D, \cE_v, \phi_v)$ is independent of $\varepsilon$, we get
    \begin{align*}
    \fP_v(D, \cE_v, \phi_v)&=\lim_{\varepsilon\to0}\ 
    \frac{\alpha_{\cE_v}(\phi_v, \phi_v)}{|2|_v\zeta_{F_v}(2)L(\frac12, \pi_v)} 
    \frac{I_{\cE_v}(\Phi_\varepsilon, \phi_v, 1)}{Z_{\cE_v}( \Phi_\varepsilon, \phi_v, \frac32)} 
    \\
    &=\frac{L(1, \eta_{\cE_v})^2|d_{\cE_v}|_v^\frac12}{2\, L(\frac12, \pi_v)}
    \cI_{\cE_v}(\phi_v, 1). \qedhere
    \end{align*}
\end{proof}

\begin{lem}\label{lem:omgarch}
Suppose that $v$ is an archimedean place.
Then we have
    \[
    \cI_{\cE_v}(\phi_v, 1)
    =c_v^{-1} L(1, \eta_{\cE_v})^{-m_v} \,
    \frac{\alpha_{\cE_v}(\phi_v, \phi_v)}{\langle \phi_v, \phi_v \rangle_v},
    \]
where 
    \[
    m_v=
        \begin{cases}
        1 & \text{if $v$ is a real place} \\
        2 & \text{if $v$ is a complex place}.
        \end{cases}
    \] 
\end{lem}

\begin{proof}
First we consider the case of a real place.
Suppose that $\cE_v$ is isomorphic to $\C$. 
Then, we have $\delta_{\cE_v}^2=-1$ by the assumption.
We have
    \begin{align*}
    \cI_{\cE_v}(\phi_v, 1)
    &=\zeta_{F_v}(2)^{-1}\frac{2}{4\pi}
    \int_{-\pi}^\pi\frac{\langle \phi_v, \pi_v(\cos\theta+(\sin\theta)\delta_{\cE_v})
    \phi_v \rangle_v}{\langle \phi_v, \phi_v \rangle_v} \d\theta \\
    &=L(1, \eta_{\cE_v})^{-1}\int_{F_v^\times\bs\cE_v^\times}
    \frac{\langle \phi_v, \pi_v(h_v)\phi_v \rangle_v}{\langle \phi_v, \phi_v \rangle_v} 
    \d h_v
    =c_v^{-1}L(1, \eta_{\cE_v})^{-1}
    \frac{\alpha_{\cE_v}(\phi_v, \phi_v)}{\langle \phi_v, \phi_v \rangle_v}.
    \end{align*}
Suppose that $\cE_v$ is isomorphic to $\R\times\R$. 
In this case, $\delta_{\cE_v}^2=1$ and hence
    \[
    \cI_{\cE_v}(\phi_v, 1)
    =\zeta_{F_v}(2)^{-1}\frac{2}{4\pi}\int_{-\pi}^\pi 
    \frac{\langle \phi_v, \pi_v(\cos\theta+(\sin\theta)\delta_{\cE_v})
    \phi_v \rangle_v}{\langle \phi_v, \phi_v \rangle_v |\cos^2\theta-\sin^2\theta|} \d\theta.
    \]
Making the substitution $z=\cot\theta$, the above integral becomes
    \begin{align*}
    \int_{F_v^\times} \frac{\langle \phi_v, \pi_v(z+\delta_{\cE_v})\phi_v \rangle_v}
    {\langle \phi_v, \phi_v \rangle_v |z^2-1|} \d z
    &=c_v^{-1}L(1, \eta_{\cE_v})^{-1}\frac{\alpha_{\cE_v}(\phi_v, \phi_v)}
    {\langle \phi_v, \phi_v \rangle_v}.
    \end{align*}

Next we treat the case where $v$ is a complex place.
In this case, $\cE_v$ is isomorphic to $\C\times\C$.
If we write an element of $K_v=\U(2)$ as 
    \[
    g=e^{\sqrt{-1}\kappa}
        \begin{pmatrix}
        e^{\sqrt{-1}(\psi+\varphi)}\cos\theta & -e^{\sqrt{-1}(\psi-\varphi)}\sin\theta \\
        e^{\sqrt{-1}(\varphi-\psi)}\sin\theta & e^{-\sqrt{-1}(\psi+\varphi)}\cos\theta
        \end{pmatrix}
    \]
with $0\leq\kappa\leq\pi$, $0\leq\theta\leq\frac\pi2$, $0\leq\varphi\leq\pi$ and $-\pi\leq\psi\leq\pi$, the normalize Haar measure on $K_v$ is given by $\d g=\frac{1}{2\pi^3}\sin2\theta\d\kappa \d\theta \d\varphi \d\psi$.
See \cite{Faraut}*{Proposition 7.2.1}.
Hence $\cI_{\cE_v}(\phi_v, 1)$ equals
    \begin{align*}
    &\frac{\zeta_{F_v}(2)^{-1}}{2\pi^2}\int_0^{\frac\pi2}\d\theta
    \int_0^\pi\d\varphi \int_{-\pi}^\pi \d\psi\
    \frac{\langle \phi_v, \pi_v(e^{\sqrt{-1}(\psi+\varphi)}\cos\theta
    +e^{\sqrt{-1}(\varphi-\psi)}\sin\theta\delta_{\cE_v})\phi_v\rangle_v\, \sin2\theta}
    {\langle \phi_v, \phi_v\rangle_v \,
    |e^{2\sqrt{-1}(\psi+\varphi)}\cos^2\theta-e^{2\sqrt{-1}(\varphi-\psi)}\sin^2\theta|_v} 
    \\[5pt]
    &=4\pi\int_{-\frac\pi2}^{\frac\pi2}\d\psi \int_0^\infty \d z\
    \frac{\langle \phi_v, \pi_v(ze^{2\sqrt{-1}\psi}+\delta_{\cE_v})\phi_v\rangle_v\, z}
    {\langle \phi_v, \phi_v\rangle_v \, |z^2e^{4\sqrt{-1}\psi}-1|_v},
    \end{align*}
where we set $z=\cot\theta$.
Note that the normalized absolute value $|\cdot|_v$ for a complex place $v$ is given by $|z|_v=z\bar{z}$,  $z\in\C$,  where $\bar{\cdot}$ is the complex conjugate.

Put $x=z\, e^{2\sqrt{-1}\psi}$.
The self-dual Haar measure on $\C$ is given by $ \d x=4z\d z \d\psi$.
Thus the last expression becomes
    \[
    \pi\int_{\C^\times}\frac{\langle \phi_v, \pi_v(x+\delta_{\cE_v})\phi_v\rangle_v}
    {\langle \phi_v, \phi_v\rangle_v}\, \frac{\d x}{|x^2-1|_v}
    =\pi^2\frac{\alpha_{\cE_v}(\phi_v, \phi_v)}{\langle \phi_v, \phi_v\rangle_v}. \qedhere
    \]
\end{proof}

\begin{cor}\label{cor:fP}
For any $v\in\Sigma$ and $\cE_v\in X(D_v)$,
    \[
    \fP_v(D, \cE_v, \phi_v)
    =\frac{|d_{\cE_v}|_v^{\frac12}L(1, \eta_{\cE_v})^{2-m_v}}{2c_v \, L(\frac12, \pi_v)} \,
    \frac{\alpha_{\cE_v}(\phi_v, \phi_v)}{\langle \phi_v, \phi_v \rangle_v}.
    \]
Here, $m_v=2$ if $v$ is a complex place and $m_v=1$ otherwise.
\end{cor}

\begin{proof}
This follows from Lemma \ref{lem:fP}, Lemma \ref{lem:omg}, Lemma \ref{lem:convarch} and Lemma \ref{lem:omgarch}.
\end{proof}

\begin{cor}\label{cor:fPun}
Suppose that $v$ is a finite place of $F$ such that $D_v$ splits, $\pi_v$ is unramified and $\phi_v$ is the normalized spherical vector. 
We also assume that $d_{\cE_v}\in\fo_v\setminus\varpi_v^2\fo_v$ and the maximal compact subgroup of $\cE_v^\times$ is contained in $g^{-1}K_vg$ for some $g\in D_v^\times$.
When $v$ is dyadic,  we further assume that $F_v=\Q_2$.
Then 
    \begin{multline*}
    \fP_v(D, \cE_v, \phi_v) =\\
    \frac{|d_{\cE_v}|_v^{\frac12}\, \zeta_{F_v}(2)L(\frac12, \pi_v\otimes\eta_{\cE_v})} 
    {2c_v L(1, \pi_v, \Ad)} \times\begin{cases} 2 & \text{if $v=2$ and $\cE_2$ is unramified over $\Q_2$,} \\ 1 & \text{otherwise.} \end{cases}    
    \end{multline*}
And $\fP_v(D, \cE_v, \phi_v)>0$.
If $v$ is not dyadic,  we have
    \begin{equation}\label{eq:0815}
    \sum_{\cE_v\in X(D_v)}N(\ff_{\cE_v})^{-1/2}\fP_v(D, \cE_v, \phi_v)
    =1-q_v^{-3}-\frac{1-q_v^{-1}}{1+q_v^{-1}}q_v^{-3}\lambda_v^2.
    \end{equation}
\end{cor}
\begin{proof}
Let $\cK_v$ denote the maximal compact subgroup of $\cE_v^\times$. 
First, we consider the case that $\cE_v$ is not a ramified extension of $F_v$. 
In this case, this is a consequence of Corollary \ref{cor:fP} and the unramified computation
    \[
    \alpha_{\cE_v}(\phi_v, \phi_v)=
    \frac{\zeta_{F_v}(2)L(\frac12, \pi_v)L(\frac12, \pi_v\otimes\eta_{\cE_v})}
    {L(1, \pi_v, \Ad)L(1, \eta_{\cE_v})} \times \vol(\overline{\cK_v})
    \]
carried out by Waldspurger \cite{Wal2}*{Lemmes 2, 3} (see also \cite{II}*{Theorem 1.2}).
When $v$ is not dyadic, we can take
    \[
    \cK_v=\left\{ \begin{pmatrix}  a & b \\ d_{\cE_v} b & a \end{pmatrix}\in \GL_2(\fo_v)\right\},
    \] 
and so we compute $\vol(\overline{\cK_v})=\int_{\cK_v} \d h_{\cE_v}=1$ by the definition \eqref{eq:localmeah} of $\d h_{\cE_v}$.  
On the other hand, when $v$ is dyadic, that is, $F_v=\Q_2$ and $d_{\cE_2}\equiv 1 \mod 4\Z_2$, we can take
    \[
    \cK_2=\left\{ 
        \begin{pmatrix}  
        a+ b/2 & b/2 \\ 
        d_{\cE_2} b/2 & a+b/2 
        \end{pmatrix}
    \; \middle|  \; a,b\in\Z_2 \right\},
    \]
since we have
    \[
    g\cK_2 g^{-1}=\left\{ \begin{pmatrix}  a & b \\ u b & a+b \end{pmatrix}\in\GL_2(\Z_2)  \right\}, \quad u=\frac{d_{\cE_2}-1}{4},\quad g=
        \begin{pmatrix}
        1&0 \\ 
        1/2& 1/2 
        \end{pmatrix}
    \]
which means $\cK_2 \subset g^{-1}\GL_2(\Z_2)g$. 
Hence, we compute $\vol(\overline{\cK_2})=\int_{\cK_2} \d h_{\cE_2}=2$ by \eqref{eq:localmeah}.

Next, we treat the case that $\cE_v$ is a ramified extension over $F_v$ and $d_{\cE_v}\in\varpi_v\fo_v^\times$. 
It is known that $\lambda_v$ equals the eigenvalue of the integral operator $\pi_v(f)$ for $\phi_v$, where $f$ is the characteristic function of $K_v\diag(1,\varpi_v)K_v$, cf. \cite{Bump}*{Proposition 4.6.6}. 
Hence, by using $\vol(K_v\diag(1,\varpi_v)K_v)=q_v+1$, we have 
    \[
    \lambda_v=\int_{K_v\diag(1,\varpi_v)K_v} \langle \phi_v,\pi_v(g_v)\phi_v \rangle_v \, \d g_v
    =(q_v+1)\, \langle \phi_v,\pi_v(\diag(1,\varpi_v))\phi_v\rangle_v.
    \]
We also note
    \[
    K_v (a+b\delta_{\cE_v}) K_v=K_v \diag(1,a^2-b^2 d_{\cE_v})K_v
    \]
if $(a,b)\in \fo_v\times \fo_v \setminus \varpi_v\fo_v\times\varpi_v\fo_v$.
Hence, we have
\begin{multline*}
\cI_{\cE_v}(\phi_v, 1)=\int_{(\fo_v^\times\times\fo_v)\sqcup(\varpi_v\fo_v\times\fo_v^\times)} \frac{\langle \phi_v,\pi_v(a+b\delta_{\cE_v})\phi_v \rangle}{|\det(a+b\delta_{\cE_v})|_v}\d a \, \d b \\ 
=(1-q_v^{-1})+(1-q_v^{-1})\langle \phi_v,\pi_v(\diag(1,\varpi_v))\phi_v\rangle = \frac{(1-q_v^{-1})(1+q_v^{-1}+q_v^{-1}\lambda_v)}{1+q_v^{-1}}
\end{multline*}
since $a+b\delta_{\cE_v}$ belongs to $K_v$ (resp. $\delta_{\cE_v} K_v$) if $(a,b)\in \fo_v^\times\times\fo_v$ (resp. $\varpi_v\fo_v\times\fo_v^\times$).
From this and Lemma \ref{lem:fP}, we obtain the desired equality. 

The remaining case is that $F_v=\Q_2$, $\cE_v$ is ramified over $F_v$, and $d_{\cE_v}\in \fo_v^\times$.   
In this case, $\fo_v=\Z_2$, $q_v=2$, $d_{\cE_v}\equiv 3 \mod 4$, and we can choose $\varpi_v=2$.
Dividing $\fo_v\times \fo_v \setminus \varpi_v\fo_v\times\varpi_v\fo_v$ into $(\fo_v^\times\times\varpi_v\fo_v)\sqcup(\varpi_v\fo_v\times\fo_v^\times)\sqcup(\fo_v^\times\times\fo_v^\times)$ we have
    \[
    \cI_{\cE_v}(\phi_v, 1)=\frac{1}{4}+\frac{1}{4}
    +\frac{1}{2}\int_{\fo^\times}\frac{\langle \phi_v,\pi_v(\diag(1,a^2-d_{\cE_v}))\phi_v \rangle}{|a^2-d_{\cE_v}|_v}\d a.
    \]
Hence, we get $\cI_{\cE_v}(\phi_v, 1)=(3+\lambda_v)/6$ since $a^2-d_{\cE_v}\in 2(1+2\fo_v)$ for $a\in\fo_v^\times$.
Therefore, we obtain the assertion by Lemma \ref{lem:fP}.

The positivity of $\fP_v(D, \cE_v, \phi_v)$ follows from Theorem \ref{thm:BB}. 
Equation \eqref{eq:0815} follows from the above calculations. 
\end{proof}


\section{Proof of the mean value theorems}\label{sec:proof}


\subsection{Proof of Theorem \ref{thm:main}}

Now we give the proof of Theorem \ref{thm:main}.

If $L(\tfrac12, \pi)=0$,  then $\cP_E(\phi)=0$ for all $E\in X(D)$ and hence the both sides of \eqref{eq:mvf} are zero.
For the rest, we assume $L(\tfrac12, \pi)\neq0$.
Under this assumption, $\sD(D, \cE_S, \phi,s)$ is not identically zero, see \cite{SW}*{Proposition 3.4}.

Let $T$ be a finite subset of $\Sigma$ containing $S$.
Set
    \[
    \sD(D, \cE_S, \phi, T, s)
    =\sum_{E\in X(D,\cE_S)}
    \frac{L(1,\eta)^2\alpha_E^{\cE_S}(\phi)\, |\cP_E(\phi)|^2\mathcal{D}_E^T(\pi,s)}
    {N(\ff_E^S)^{s-1}}
    \]
and $\fC(D, \cE_S, \phi, T)=\underset{s=3/2}{\Res}\sD(D, \cE_S, \phi, T,s)$.
Clearly we have $\sD(D, \cE_S, \phi, S, s)=\sD(D, \cE_S, \phi, s)$ and $\fC(D,\cE_S,\phi,S)=\fC(D,\cE_S,\phi)$.
For $\cE_{T'}=(\cE_v)_{v\in T'}\in X(D_{T'})=\prod_{v\in T'}X(D_v)$, we set $N(\ff_{\cE_{T'}})=\prod_{v\in T'}N(\ff_{\cE_v})$. 
By Lemma \ref{lem:unchange}, we obtain
    \[
    \sD(D, \cE_S, \phi, T, s)=\sum_{\cE'\in X(D_{T\setminus S})} 
    \frac{1}{N(\ff_{\cE'})^{s-1}}\sD(D, \cE_S\cup\cE', \phi, s),  
    \]
where $\cE'=(\cE'_v)_{v\in T\setminus S}$ runs through $X(D_{T\setminus S})$ and $\cE_S\cup\cE'$ is an element of $X(D_T)=\prod_{v\in T}X(D_v)$.
In particular,   $\sD(D, \cE_S, \phi, T, s)$ is a meromorphic function on $\C$ and has a simple pole at $s=\tfrac32$ with a positive residue (see Lemma \ref{lem:dir} and Corollary \ref{cor:nonv3}).
Together with Theorem \ref{thm:re} and Corollary \ref{cor:nonv3}, we get
    \[
    \fC(D, \cE_S, \phi,T)
    =\fc_F c_F\langle \phi, \phi \rangle\cdot \frac{\zeta_F(2)L(\frac12, \pi)}{L^T(2, \pi, \Ad)}
    \sum_{\cE'\in X(D_{T\setminus S})}
    \frac{\fP_T(D, \cE_S\cup\cE', \phi_T)}
    {N(\ff_{\cE'})^{\frac12}}.
    \]
By \eqref{eq:0815}, this becomes
    \begin{equation}\label{eq:20200904}
    \fC(D, \cE_S, \phi, T)=\fC(D, \cE_S, \phi) \prod_{v\in T\setminus S}
    L(2,\pi_v,\mathrm{Ad})
    \left\{ 1-q_v^{-3}-\frac{1-q^{-1}}{1+q^{-1}}q_v^{-3}\lambda_v^2\right\}.
    \end{equation}

We follow the same line as the proof of \cite{KY1}*{Theorem 6.22} to deduce \eqref{eq:mvf} from \eqref{eq:20200904}.
In particular, we apply the filtering process formulated by Datskovsky and Wright.
Assume that $T$ contains all places above primes less than $20$.
We write
    \[
    \sD(D, \cE_S, \phi, T, s+\tfrac{3}{2}-t)=\sum_{m=1}^\inf \frac{c_{m}(t)}{m^s}
    \]
and
    \[
    \mathcal{D}_E^T(\pi,s+\tfrac{3}{2}-t)=\sum_{n=1}^\inf \frac{ a_n(E,\pi,T,t) }{n^s}.
    \]
Note that $a_1(E,\pi,T,t)=1$.
From the definition of $\sD(D, \cE_S, \phi, T, s)$ we have
    \[
    c_m(t)=\sum_{n=1}^\infty \sum_{\substack{E\in X(D,\cE_S), \\   nN(\ff^S_E)=m}}  
    L(1,\eta)^2 \,\alpha_E^{\cE_S}(\phi)\, |\cP_E(\phi)|^2\,  
    \frac{a_n(E,\pi,T,t)}{N(\ff_E^S)^{\frac12-t}}.
    \]
The following theorem is called the Tauberian theorem.
\begin{thm}\label{thm:tauberian}
    Let $a_n$ be a non-negative real number for any $n\in\N$, and $M$ be a positive real number. 
    A Dirichlet series $L(s) = \sum_{n=1}^\inf \frac{a_n}{n^{-s}}$ is absolutely and uniformly convergent on any compact set in the domain $\{ s\in \C \mid \Re(s)>M \}$. In addition, $L(s)$ is meromorphically continued to a domain including $\{ s\in \C \mid \Re(s)\ge M \}$, and $L(s)$ has a simple pole only at $s = M$. Set $A=\underset{s=\frac32}{\Res}\, L(s)>0$. Then, we have
    \[
    \lim_{X\to\inf }X^{-M} \sum_{n<X}a_n=\frac{A}{M}.
    \]
\end{thm}
\begin{proof}
For the case $M=1$, we refer to \cite{Nark}*{Theorem I in p.464}. 
Our assertion is obtained by putting $\tilde{L}(s)=L(Ms)$.
\end{proof}
Applying Theorem \ref{thm:tauberian} to $\sD(D, \cE_S, \phi, T, s+\tfrac{3}{2}-t)$, we obtain
    \begin{equation}\label{eq:tauberian}
    \lim_{x\to\inf} x^{-t} \sum_{m<x}c_m(t)
    =\frac{1}{t}\fC(D, \cE_S, \phi, T). 
    \end{equation}

By the assumption for $T$, we have $2\,\left| q_v^{-\frac{25}{64}}+q_v^{-\frac{39}{64}} \right|< 0.9$ for any $v\notin T$.
Hence, by Theorem \ref{thm:BB} and the argument in the proof of Lemma \ref{lem:propzeta}, we find for $v\notin T$ 
    \[
    \mathcal{D}_{E_v}(\pi_v,s+\tfrac{3}{2}-t)=1+a_v q_v^{-2s-2+2t}+b_vq_v^{-4s-5+4t}
    \]
with some $0\leq a_v\leq 2$ and $0\leq b_v\leq 1$ for $v\notin T$.
It follows from this fact that $a_n(E,\pi,T,t)\geq 0$ and hence we get
    \begin{multline*}
     \sum_{\substack{E\in X(D, \cE_S), \\ N(\ff^S_E)<x}}  
     L(1,\eta)^2 \, \alpha_E^{\cE_S}(\phi)\, |\cP_E(\phi)|^2 \, N(\ff_E^S)^{t-\tfrac{1}{2}}\\
    \leq\sum_{n=1}^\infty \sum_{\substack{E\in X(D,\cE_S), \\   nN(\ff_E^S)<x}}  
    L(1,\eta)^2 \, \alpha_E^{\cE_S}(\phi)\, |\cP_E(\phi)|^2\,  
    \frac{a_n(E,\pi,T,t)}{N(\ff_E^S)^{\frac12-t}} 
    =\sum_{m<x}c_m(t).
    \end{multline*}
Combining this inequality with \eqref{eq:20200904} and \eqref{eq:tauberian}, and taking $\ds\lim_{T\to\Sigma}$, we obtain the upper bound
    \begin{equation}\label{eq:limsup}
    \limsup_{x\to\infty}\ x^{-t}
    \sum_{\substack{E\in X(D, \cE_S),\\  N(\ff_E^S)<x}}
    N(\ff_E^S)^{t-\tfrac{1}{2}}
    \, L(1,\eta)^2 \, \alpha_E^{\cE_S}(\phi)\, |\cP_E(\phi)|^2\leq C(D, \cE_S, \phi, t),
    \end{equation}
where we set
    \begin{align*}
    C(D, \cE_S, \phi, t)&=\lim_{T\to\Sigma}\frac{1}{t}\fC(D, \cE_S, \phi, T) \\
    &=\frac{1}{t}\fC(D,\cE_S, \phi) L^S(2,\pi,\mathrm{Ad})  
    \prod_{v\notin S}\left\{ 1-q_v^{-3}-\frac{1-q_v^{-1}}{1+q_v^{-1}}q_v^{-3}
    \lambda_v^2  \right\} .
    \end{align*}
In particular,  there is a constant $R>0$ such that for any $x>0$,
    \begin{equation}\label{eq:limsup2}
    \sum_{\substack{E\in X(D, \cE_S),\\  N(\ff^S_E)<x}}
    N(\ff_E^S)^{t-\tfrac{1}{2}}
    \, L(1,\eta)^2 \, \alpha_E^{\cE_S}(\phi)\, |\cP_E(\phi)|^2
    \leq R\cdot x^t C(D, \cE_S, \phi, t).
    \end{equation}
    
We set $B_v(s,t)\coloneqq1+2q_v^{-2s-2+2t}+q_v^{-4s-5+4t}$,  $B^T(s, t)\coloneqq\prod_{v\not\in T}B_v(s,t)$
and write $B^T(s,t)=\sum_{n=1}^\infty b_n(T,  t)n^{-s}$.
Since $a_n(E, \pi, T, t)\leq b_n(T, t)$, we have
    \begin{multline*}
    \sum_{m<x}c_m(t)-     
    \sum_{\substack{E\in X(D,  \cE_S), \\ N(\ff_E^S)<x}} 
    N(\ff_E^S)^{t-\tfrac{1}{2}}\, L(1,\eta)^2 \, 
    \alpha_E^{\cE_S}(\phi)\, |\cP_E(\phi)|^2 \\
    =\sum_{n=2}^\infty \sum_{\substack{E\in X(D,\cE_S), \\  nN(\ff_E^S)<x}} 
    L(1,\eta)^2 \, \alpha_E^{\cE_S}(\phi)\, |\cP_E(\phi)|^2\,   
    \frac{a_n(E,\pi,T,t)}{N(\ff_E^S)^{\frac12-t}} \\
    \leq \sum_{n=2}^\infty \sum_{\substack{E\in X(D,\cE_S), \\  
    nN(\ff_E^S)<x}} 
    L(1,\eta)^2 \, \alpha_E^{\cE_S}(\phi)\, |\cP_E(\phi)|^2\,   
    \frac{b_n(T, t)}{N(\ff_E^S)^{\frac12-t}} \\
    = \sum_{n=2}^\inf b_n(T,t) 
    \sum_{\substack{E\in X(D,\cE_S), \\  N(\ff^S_E)<\frac{x}{n}}} 
    N(\ff_E^S)^{t-\tfrac{1}{2}}\, L(1,\eta)^2 \, 
    \alpha_E^{\cE_S}(\phi)\, |\cP_E(\phi)|^2 .
    \end{multline*}
It follows from the upper bound \eqref{eq:limsup2} that the last expression is bounded above by
    \[
    R\sum_{n=2}^\inf b_n(T,t) \left(\frac{x}{n}\right)^t C(D, \cE_S, \phi, t)
    =R\cdot x^t\, C(D, \cE_S, \phi, t)(B^T(t,t)-1).
    \]
We deduce from this and \eqref{eq:tauberian} that
    \begin{align*}
    &\liminf_{x\to\inf}\ x^{-t}\sum_{\substack{E\in X(D,\cE_S), \\  N(\ff_E^S)<x}}    
    N(\ff_E^S)^{t-\tfrac{1}{2}}\,  L(1,\eta)^2 \, \alpha_E^{\cE_S}(\phi)\, |\cP_E(\phi)|^2\\
    &\hspace{45pt}\geq \liminf_{x\to\infty}\ x^{-t}
    \sum_{m<x}c_m(t)-R\cdot C(D, \cE_S, \phi, t)(B^T(t, t)-1) \\
    &\hspace{65pt}=\frac{1}{t}\fC(D, \cE_S, \phi, T) 
    -R\cdot C(D, \cE_S, \phi, t)\, (B^T(t,t)-1).
    \end{align*}
Since $\ds\lim_{T\to\Sigma}B^T(t,t)=1$,   taking $\ds\lim_{T\to\Sigma}$ of the both sides of the above inequality,   we get
    \begin{equation}\label{eq:liminf}
    \liminf_{x\to\inf}\ x^{-t}\sum_{\substack{E\in X(D,\cE_S), \\ N(\ff_E^S)<x}}
    N(\ff_E^S)^{t-\tfrac{1}{2}}\,    
    L(1,\eta)^2 \, \alpha_E^{\cE_S}(\phi)\, |\cP_E(\phi)|^2\geq  C(D, \cE_S, \phi, t).
    \end{equation}

From \eqref{eq:re}, \eqref{eq:limsup} and \eqref{eq:liminf}, we obtain
    \begin{align*}
    &\hspace{-25pt}\lim_{x\to\infty} 
    x^{-t}\sum_{\substack{E\in X(D, \cE_S) \\ N(\ff_S^E)<x}}
    N(\ff_E^S)^{t-\frac12}L(1, \eta)^2\alpha_E^{\cE_S}(\phi)|\cP_E(\phi)|^2 \\
    &=C(D, \cE_S, \phi, t)=\frac{1}{t}\fC(D,\cE_S, \phi) L^S(2,\pi,\mathrm{Ad})  
    \prod_{v\notin S}\left\{ 1-q_v^{-3}-\frac{1-q_v^{-1}}{1+q_v^{-1}}q_v^{-3}
    \lambda_v^2  \right\} \\
    &=L(\tfrac12, \pi) \frac{2\pi^{r_2}c_F^2\langle \phi, \phi \rangle}
    {t|\Delta_F|^\frac12}
    \fP_S(D, \cE_S, \phi_S)
    \prod_{v\notin S}\left\{1-q_v^{-3}-\frac{q_v-1}{q_v+1}q_v^{-3}\lambda_v^2\right\}.
    \end{align*}
Hence, we deduce
    \begin{align*}
    & \lim_{x\to\inf} \, x^{-t}                    
    \sum_{\substack{E\in X(D,\cE_S),\\ N(\ff_E^S)<x}} 
     N(\ff_E^S)^{t-1} \, L(1,\eta_E)^2 \, \alpha_{E, S}(\phi)^{-1}|\cP_E(\phi)|^2 \\
    & =L(\tfrac12, \pi) \frac{2\pi^{r_2}c_F^2}
    {t|\Delta_F|^\frac12}
    \times\prod_{v\in S}\frac{L(1, \eta_{\cE_v})^{2-m_v}}{2c_v \, L(\frac12, \pi_v)}
    \times\prod_{v\notin S}\left\{1-q_v^{-3}-\frac{q_v-1}{q_v+1}q_v^{-3}\lambda_v^2\right\} ,
    \end{align*}
since we have 
\[
\alpha_{E,S}(\phi,\phi)^{-1}|\cP_E(\phi)|^2=N(\ff_E^S)^{1/2}\alpha_E^{\cE_S}(\phi)|\cP_E(\phi)|^2 \times \prod_{v\in S}|d_{\cE_v}|^{-\frac12}\alpha_{\cE_v}(\phi_v,\phi_v)^{-1}
\]
and by Corollary \ref{cor:fP} 
\[
\fP_S(D, \cE_S, \phi_S)\times \prod_{v\in S}|d_{\cE_v}|^{-\frac12}\alpha_{\cE_v}(\phi_v,\phi_v)^{-1}=\frac{1}{\langle \phi, \phi \rangle}\prod_{v\in S}\frac{L(1, \eta_{\cE_v})^{2-m_v}}{2c_v \, L(\frac12, \pi_v)}.
\]
From this we obtain Theorem \ref{thm:main}.


\appendix
\def\thesection{\Alph{section}}
\section{Numerical examples for mean value theorems}\label{sec:appendix}
\begin{center}
by \textsc{M. Suzuki, S. Wakatsuki and S. Yokoyama\footnote{\address{Shun'ichi Yokoyama \\
Department of Mathematical Sciences\\
Graduate School of Science, Tokyo Metropolitan University\\
1-1 Minami-Osawa, Hachioji-shi, Tokyo, 192-0397, JAPAN}
\email{s-yokoyama@tmu.ac.jp}}}
\end{center}

Using \texttt{Magma}, we provide numerical examples for Theorem \ref{thm:mvfhol} and Theorem \ref{thm:mvfalg}.


\subsection{Elliptic modular forms}

We briefly recall the notation of \S\,\ref{sec:mvfhol}.
Let $k$ be a non-negative even integer and $f\in S_{k}(\SL_2(\Z))$ be a weight $k$ cuspidal Hecke eigenform with Fourier expansion $ f(z)=\sum_{n=1}^\infty a_nq^n$,  where $q=e^{2\pi\sqrt{-1}z}$.
Set $\lambda_p=p^{1-\frac{k}{2}}a_p$ for each prime $p$.
Let $\pi=\otimes_v\pi_v$ be the cuspidal automorphic representation of $\GL_2(\A_\Q)$ corresponding to $f$.

Let $E$ be a real quadratic field with discriminant $\fd=\fd_E$.
We have an equality of meromorphic functions $L_\fin(s, \pi\otimes\eta)=L(s+\tfrac{k-1}{2}, f, \fd)$.
Here, $\eta$ is the quadratic character on $\A_\Q^\times$ corresponding to $E$ and the right hand side is the analytic continuation of the Dirichlet $L$-series $\sum_{n=1}^\infty\left(\frac{\fd}{n}\right)a_n n^{-s-\frac{k-1}{2}}$, where $\left(\frac{\fd}{n}\right)$ is the Kronecker symbol.
It is easy to see that the function $f_\fd=\sum_{n=1}^\infty \left(\frac{\fd}{n}\right)a_nq^n$ belongs to $S_k(\Gamma_0(\fd^2))$ and satisfies
    \[
    f_\fd\left(\frac{-1}{\fd^2z}\right)=\fd^k z^k f_\fd(z).
    \]
The $L$-function $L(s, f, \fd)$ has an integral expression
    \begin{align*}
    (2\pi)^{-s}\Gamma(s)L(s, f, \fd)&=\int_0^\infty f_\fd(\sqrt{-1}y)y^s \frac{\d y}{y} \\
    &=\int_{1/\fd}^\infty 
    f_\fd(\sqrt{-1}y)(y^s+(-1)^{\frac k2}\fd^{k-2s}y^{k-s}) \frac{\d y}{y} .
    \end{align*}
The last integral converges for all $s$. Hence we get
    \[
    (2\pi)^{-\frac{k}{2}}L(\tfrac{k}{2}, f, \fd)
    =(1+(-1)^{\frac{k}{2}}) \sum_{n=1}^\infty \left(\frac\fd n\right) a_n 
    \Xi_k\left(\frac{2\pi n}{\fd}\right),
    \]
where 
    \[
    \Xi_k(x)=(\fd x)^{-\frac{k}{2}}e^{-x}\sum_{m=0}^{\frac{k}{2}-1}\frac{x^m}{m!}.
    \]

Now we give a numerical example for $k=12$, the first non-trivial case.
The space $S_{12}(\SL_2(\Z))$ is $1$-dimensional and generated by the Hecke eigenform
    \[
    \Delta(z)=q\prod_{n=1}^\infty(1-q^n)^{24}=\sum_{n=1}^\infty \tau(n)q^n.
    \]
By \cite{Zagier}*{(29)}, the inner product $\langle \Delta, \Delta \rangle$ is written as an infinite series and we obtain the numerical value
    \[
    \langle \Delta, \Delta \rangle = 
    1.035 \ 362 \ 056 \ 804 \ 320 \ 922 \ 347 \ 816 \times10^{-6}.
    \]

For $x>0$, put
    \[
    M^\mathrm{ell}(x)=x^{-1}\sum_{\substack{E\in X(\M_2(\Q), \cE_S) \\ a_E<x}}
    L(\tfrac12, \pi\otimes\eta_E)
    =\frac{2 \, \Gamma(6)}{(2\pi)^6 \, x}
    \sum_{\substack{E\in X(\M_2(\Q), \cE_S) \\ a_E<x}} L(6, \Delta, \fd_E).
    \]
Here, $S=\{2, \infty\}$, $\cE_2\in X(\M_2(\Q_2))$ is arbitrary, $\cE_\infty$ is isomorphic to $\R\times\R$ and $a_E$ is defined in \eqref{eq:a_E}.
The following table shows the values $M^\mathrm{ell}(x)$ for $x=i\cdot 10^5$, $i=1, 2,\ldots, 10$.
The last row is the values of the right hand side of \eqref{eq:mvfhol1} for $t=1$.

\begin{table}[htb]\renewcommand{\arraystretch}{1.3}
  \begin{tabular}{|c|c|c|c|c|} \hline
            & $\cE_2\simeq\Q_2\times\Q_2$ & $\cE_2\simeq\Q_2(\sqrt{5})$ & $\cE_2\simeq\Q_2(\sqrt{3})$ & $\cE_2\simeq\Q_2(\sqrt{7})$ \\ \hline\hline
  $M^\mathrm{ell}(10^5)$ & 0.0002521751 & 0.0004181237 & 0.0004732234 & 0.0004706102 \\ \hline
  $M^\mathrm{ell}(2\cdot 10^5)$ & 0.0002524130 & 0.0004201329 & 0.0004729765 & 0.0004710201 \\ \hline
  $M^\mathrm{ell}(3\cdot 10^5)$ & 0.0002523719 & 0.0004194494 & 0.0004727592 & 0.0004729393 \\ \hline
  $M^\mathrm{ell}(4\cdot 10^5)$ & 0.0002520415 & 0.0004198913 & 0.0004731776 & 0.0004729414  \\ \hline
  $M^\mathrm{ell}(5\cdot 10^5)$ & 0.0002522972 & 0.0004196702 & 0.0004732125 & 0.0004720879 \\ \hline
  $M^\mathrm{ell}(6\cdot 10^5)$ & 0.0002520912 & 0.0004197186 & 0.0004728869 & 0.0004720721 \\ \hline
  $M^\mathrm{ell}(7\cdot 10^5)$ & 0.0002520625 & 0.0004196767 & 0.0004727840 & 0.0004723050 \\ \hline
  $M^\mathrm{ell}(8\cdot 10^5)$ & 0.0002521590 & 0.0004199355 & 0.0004729952 & 0.0004728795 \\ \hline
  $M^\mathrm{ell}(9\cdot 10^5)$ & 0.0002519507 & 0.0004204256 & 0.0004732688 & 0.0004730878  \\ \hline
  $M^\mathrm{ell}(10^6)$ & 0.0002519197 & 0.0004203996 & 0.0004733662 & 0.0004729954 \\ \hline\hline
  RHS & 0.0002520826 & 0.0004201377 & 0.0004726550 & 0.0004726550 \\\hline   
 \end{tabular}
\end{table}


\subsection{Algebraic modular forms}\label{app:mvfalg}

Let $D$ be the quaternion division algebra over $\Q$ with discriminant $q=11$.
We take $\alpha, \beta \in D$ and a maximal order $\fO$ as in \S\,\ref{sec:mvfalg}.
Then $|\mathrm{Cl}(\fO)|=2$.
Take a set of coset representatives $1=x_1, x_2\in PD_{\A_\fin}^\times$ of $PD^\times\bs PD_\A^\times/PD_\infty^\times U$.
The space of algebraic modular forms $\mathcal{A}(\fO)$ is 2-dimensional.
Let $\phi$ be a non-zero element of $\mathcal{A}(\fO)$ which is orthogonal to the constant functions:
    \[
    0=(\phi, \mathbf{1})=\frac{\phi(x_1)}{w_1}+\frac{\phi(x_2)}{w_2},
    \]
where $\mathbf{1}$ is the constant function such that $\mathbf{1}(x_1)=\mathbf{1}(x_2)=1$ and $w_i=|PD^\times\cap x_jUx_j^{-1}|$.
As in \S \ref{sec:mvfalg}, let $\pi$ be the corresponding automorphic representation.
By Eichler mass formula,  we know that $\{w_1, w_2\}=\{2, 3\}$.
One can compute the Hecke operators $T_p$ on $\mathcal{A}(\fO)$ using Brandt matrices in \texttt{Magma} (see \cite{Voight}*{\S 41} for Brandt matrices). 
For example,  we have
    \[
    T_2=
        \begin{pmatrix}
        1 & 2 \\
        3 & 0
        \end{pmatrix}
    \]
with respect to the basis $\phi_1$, $\phi_2$ defined as $\phi_i(x_j)=\delta_{ij}$ ($\delta_{ij}$ means the Kronecker delta), that is, we have $(T_2\phi_1,T_2\phi_2)=(\phi_1,\phi_2)T_2$. 
Computing the eigenvectors of this Hecke operator,  we see that $\phi=2\phi_1-3\phi_2\in\mathcal{A}(\fO)$ is a Hecke eigenform.  
Thus we get $(w_1, w_2)=(2, 3)$, since $\phi(x_1)=2$ and $\phi(x_2)=-3$. 
Note that we also get $\lambda_2=-2$ by $T_2 \left(\begin{smallmatrix} 2 \\ -3 \end{smallmatrix} \right)=\lambda_2\left(\begin{smallmatrix} 2 \\ -3 \end{smallmatrix} \right)$.

We write 
    \[
    \sum_{t\in\mathrm{Cl}(E)}\phi(t)=n_1(E)\phi(x_1)+n_2(E)\phi(x_2),
    \]
where $n_i(E)$ is the number of $t\in\mathrm{Cl}(E)$ which is sent to $x_i$ under the map \eqref{eq:classmap}.

For $x>0$, put
    \[
    M^\mathrm{alg}(x)=x^{-\frac32}\sum_{\substack{E\in X(D, \cE_S)\\ a_E<x}}
     \left| \sum_{t\in\mathrm{Cl}(E)}\phi(t) \right|^2
    =x^{-\frac32}\sum_{\substack{E\in X(D, \cE_S)\\ a_E<x}} (2n_1(E)-3n_2(E))^2.
    \] 
Here, $S=\{2, 11, \infty\}$, $\cE_2\in X(D_2)$ is arbitrary, $\cE_{11}\simeq\Q_{11}(\sqrt{2})$, $\cE_\infty\simeq\C$ and $a_E$ is defined in \eqref{eq:a_E}. 
The following table shows the values $M^\mathrm{alg}(x)$ for $x=i\cdot 10^5$, $i=1, 2, \ldots, 10$.
The last row is the value of the right hand side of \eqref{eq:mvfalg}.

\begin{table}[htb]\renewcommand{\arraystretch}{1.3}
  \begin{tabular}{|c|c|c|c|c|} \hline
& $\cE_2\simeq\Q_2\times\Q_2$ & $\cE_2\simeq\Q_2(\sqrt{5})$ & $\cE_2\simeq\Q_2(\sqrt{3})$ & $\cE_2\simeq\Q_2(\sqrt{7})$ \\ \hline\hline
  $M^\mathrm{alg}(10^5)$ &0.0046337803
  &0.0231138780&0.0231920179&0.0231049920 \\ \hline
  $M^\mathrm{alg}(2\cdot 10^5)$ & 0.0046185984
   &0.0230698711&0.0231302785&0.0230117557 \\ \hline
  $M^\mathrm{alg}(3\cdot 10^5)$ & 0.0046482718
   &0.0231826737&0.0231598032&0.0230706705 \\ \hline
  $M^\mathrm{alg}(4\cdot 10^5)$ & 0.0046135180
   &0.0231018692&0.0231416705&0.0232480416 \\ \hline
  $M^\mathrm{alg}(5\cdot 10^5)$ & 0.0046270098
   &0.0232145560&0.0231038542&0.0231405163 \\ \hline
  $M^\mathrm{alg}(6\cdot 10^5)$ & 0.0046187112
   &0.0231879730&0.0233013954&0.0231419211 \\ \hline
  $M^\mathrm{alg}(7\cdot 10^5)$ & 0.0046277835
   &0.0231512776&0.0232556126&0.0230933603 \\ \hline
  $M^\mathrm{alg}(8\cdot 10^5)$ & 0.0046193685
   &0.0231590930&0.0231640137&0.0230301054 \\ \hline
  $M^\mathrm{alg}(9\cdot 10^5)$ & 0.0046396704
   &0.0231281141&0.0230524326&0.0232173711 \\ \hline
  $M^\mathrm{alg}(10^6)$  & 0.0046304910 & 0.0231243940 & 0.0230294240 & 0.0231477800 \\ \hline\hline
  RHS & 0.0046263724 & 0.0231318618 & 0.0231318618 & 0.0231318618 \\ \hline
  \end{tabular}
\end{table}


\end{document}